\documentclass[12pt]{amsart}
\usepackage{amssymb}
\usepackage{amsfonts}
\usepackage{amscd}
\usepackage{latexsym}
\usepackage{enumerate}
\usepackage{amsmath}
\usepackage{xy} \xyoption{all}

\numberwithin{equation}{section}

\newcounter{cs}
\stepcounter{cs}
\newcounter{ds}
\stepcounter{ds}

\newcommand{\casos}{\begin{itemize}}
\newcommand{\fcasos}{\end{itemize}\setcounter{cs}{1}}

\newcommand{\famm}[2]{\left({#1}\mid{#2}\right)}

\newcommand{\ol}{\overline}

\newtheorem{lem}{Lemma}[section]

\newtheorem{theor}[lem]{Theorem}
\newtheorem*{theor*}{Theorem}
\newtheorem{prop}[lem]{Proposition}

\theoremstyle{definition}
\newtheorem{defi}[lem]{Definition}

\newtheorem{exem}[lem]{Example}

\newtheorem{remark}[lem]{Remark}

\newcommand{\setm}[2]{\{{#1}\mid{#2}\}}

\DeclareMathOperator{\rL}{L}

\newcommand{\dnw}{\mathbin{\downarrow}}

\newfont{\gd}{eufm10 scaled \magstep1}
\newfont{\gs}{eufm7 scaled \magstep1}
\newfont{\gss}{eufm5 scaled \magstep1}
\newcommand{\Gbd}[1]{\mbox{\gd #1}}
\newcommand{\Gbs}[1]{\mbox{\gs #1}}
\newcommand{\Gbss}[1]{\mbox{\gss #1}}
\newcommand{\got}[1]{\mathchoice{\Gbd #1}{\Gbd #1}{\Gbs #1}{\Gbss #1}}

     \newcommand{\mon}[1]{\mathcal{V}(#1)}              
     
     \newcommand{\cb}[0]{K}                             

\begin{document}
\title[The regular algebra of a poset]{The regular algebra of a poset}
\author{Pere Ara}\address{Departament de Matem\`atiques, Universitat Aut\`onoma de
Barcelona, 08193, Bellaterra (Barcelona),
Spain}\email{para@mat.uab.cat}

\thanks{Partially supported by the DGI
and European Regional Development Fund, jointly, through Project
MTM2005-00934, and by the Comissionat per Universitats i Recerca de
la Generalitat de Catalunya. } 
\date{\today}

\begin{abstract}
Let $K$ be a fixed field. We attach to each finite poset $\mathbb P$
a von Neumann regular $K$-algebra $Q_K(\mathbb P)$ in a functorial
way. We show that the monoid of isomorphism classes of finitely
generated projective $Q_K(\mathbb P)$-modules is the abelian monoid
generated by $\mathbb P$ with the only relations given by $p=p+q$
whenever $q<p$ in $\mathbb P$. This extends the class of monoids for
which there is a positive solution to the realization problem for
von Neumann regular rings.
\end{abstract}
\maketitle

\section*{Introduction}
\label{sect:intro}

An old theorem of Jacobson \cite{Jacobson} states that, given a
field $K$, the algebra $A=K\langle x,y:yx=1\rangle $ has the
following uniqueness property: Given any other $K$-algebra $R$ with
elements $a,b$ such that $ba=1$ and $ab\ne 1$, the unique
$K$-algebra homomorphism  $A\to R$ sending $x$ to $a$ and $y$ to
$b$, is one-to-one. A very natural representation of $A$ is the one
given by the algebraic analogue of the Toeplitz algebra. This is
defined as the subalgebra of $E=\text{End}_K(K[z])$, generated by
the unilateral shift  $b$, given by multiplication by $z$, and the
endomorphism $a\in E$ defined by $(z^i)a=z^{i-1}$ for $i\ge 1$ and
$(1)a=0$. (Note that here endomorphisms act on the right of their
arguments.) Clearly, $ba=1$, but $1-ab$ is the projection onto the
one-dimensional subspace $1\cdot K$ of $K[z]$, with kernel $zK[z]$.

Let $\psi \colon A\to E$ be the unique $K$-algebra homomorphism
sending $x$ to $a$ and $y$ to $b$. For any polynomial $f\in K[x]$
such that $f(0)\ne 0$, the image $\psi (f)$ is invertible in $E$,
because the power series defining $f^{-1}$ is convergent in $E$. It
follows that there is a unique homomorphism $\ol{\psi}\colon
A\Sigma^{-1}\to E$ extending $\psi$. Here the algebra $A\Sigma
^{-1}$ is the {\it universal localization} of $A$ with respect to
$\Sigma$, that is, the algebra obtained from $A$ by formally
inverting all the polynomials in $\Sigma$, see \cite{free} and
\cite{scho}. It turns out that the algebra $A\Sigma ^{-1}$ is a von
Neumann regular ring, and that the map $\ol{\psi}$ is also
injective, thus providing a concrete faithful representation of it.
The algebra $A\Sigma ^{-1}$ can be identified with the algebra
$Q_K(E_1)$ of \cite{AB2}, associated with the quiver $E_1$ described
below.

\begin{figure}[htb]
 \[
 {
 \xymatrix{
 & v_{1}\ar@(ul,ur)\ar[d]  \\
 & v_{0}
 }
 }
 \]
\caption{The quiver $E_1$.} \label{Fig:QuiverS1}
\end{figure}
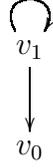

An algebra similar to $A\Sigma ^{-1}$ was used in \cite{MM} to
give the first example of a von Neumann regular ring with stable
rank $2$.

The purpose of this paper is to construct a new class of
$K$-algebras, yielding a wide generalization of the above
(algebraic) Toeplitz algebras. For each finite poset $\mathbb P$, we
will construct a $K$-algebra $Q_K(\mathbb P)$, in such a way that
the one corresponding to the poset $\mathbb P =\{q,p \}$, where the
only non-trivial relation is given by $q<p$, is precisely the
algebra $Q_L(E_1)$, with $L=K(t_1,t_2, \dots )$ a purely
transcendental extension of $K$, and $E_1$ is the quiver described
before. (For technical reasons, it is convenient to have such an
infinite number of variables at our disposal.) In general, there is
a natural faithful representation of the algebra $Q_K(\mathbb P)$ on
a vector space $V(\mathbb P)$, which is given locally by Toeplitz
operators (Theorem \ref{theor:repres}).

In order to put the construction in a wider perspective we need some
preliminary definitions. For any ring $R$, the monoid $\mon{R}$ of
isomorphism classes of finitely generated projective $R$-modules is
always a {\it conical monoid}, that is, whenever $x+y=0$, we have
$x=y=0$. Recall that an {\it order-unit} in a monoid $M$ is an
element $u$ in $M$ such that for every $x\in M$ there is $y\in M$
and $n\ge 1$ such that $x+y=nu$. Observe that, if $R$ is a unital
ring, then $[R]$ is a canonical order-unit in $\mon{R}$. By results
of Bergman \cite[Theorems 6.2 and 6.4]{Bergman} and Bergman and
Dicks \cite[page 315]{BD}, any conical monoid with an order-unit
appears as $\mathcal V (R)$ for some unital hereditary ring $R$.

For a von Neumann regular ring $R$, the monoid $\mon{R}$  is in
addition a refinement monoid. Recall that an abelian monoid $M$ is a
{\it refinement monoid} in case any equality $x_1+x_2=y_1+y_2$
admits a refinement, that is, there are $x_{ij}$, $1\le i,j \le 2$
such that $x_i=x_{i1}+x_{i2}$ and $y_j=x_{1j}+x_{2j}$ for all $i,j$,
see e.g. \cite{AMP}. It is an outstanding open problem to decide
whether all countable, conical refinement monoids can be represented
as monoids $\mon{R}$ for a von Neumann regular ring $R$, see
\cite{directsum}, \cite{AB2}. It was shown by Wehrung in
\cite{Wehisrael} that there are conical refinement monoids of size
$\aleph _2$ which cannot be represented. We refer the reader to
\cite{Areal} for a recent survey on this problem.

Note that if a monoid $M$ is realizable by a von Neumann regular
ring $R$, i.e. $\mon{R}\cong M$, and $M$ has an order-unit $u$, then
there is an idempotent $e\in M_n(R)$, for some $n\ge 1$, which
corresponds to $u$ through the isomorphism, and then $eM_n(R)e$ is a
unital von Neumann regular ring realizing $M$.

A class of monoids whose members are expected to be realizable is
the class of finitely generated conical refinement monoids. These
monoids enjoy good monoid-theoretic properties, the most important
being that they are primely generated \cite[Corollary 6.8]{Brook01}.
Recall that every monoid $M$ is endowed with a natural pre-order,
the so-called {\it algebraic pre-order}, given by $x\le y$ if and
only if there is $z\in M$ such that $y=x+z$. A {\it prime element}
in an (abelian) monoid $M$ is an element $p$ such that $p\nleq 0$
and, whenever $p\le a+b$, then either $p\le a$ or $p\le b$. The
monoid $M$ is said to be {\it primely generated} in case every
element of $M$ is a sum of primes. Note that a finitely generated
monoid has always an order-unit, namely the sum of a finite
generating set.

We say that $M$ is {\it antisymmetric} in case the algebraic
pre-order $\le$ is actually a partial order. Observe that every
antisymmetric monoid is conical.

A {\it primitive} monoid is a primely generated antisymmetric
refinement monoid. A primitive monoid is completely determined by
its set of primes $\mathbb P (M)$ together with a transitive and
antisymmetric relation $\lhd$ on it, given by $$q\lhd p \iff
p+q=p.$$ Indeed given such a pair $(\mathbb P, \lhd )$, the abelian
monoid $M(\mathbb P,\lhd)$ defined by taking as a set of generators
$\mathbb P$ and with relations given by $p=p+q$ whenever $q\lhd p$,
is a primitive monoid, and the correspondences $M\mapsto (\mathbb P
(M),\lhd) $ and $(\mathbb P,\lhd)\mapsto M(\mathbb P,\lhd)$ provide
a bijection between isomorphism types of primitive monoids and
isomorphism types of pairs $(\mathbb P,\lhd)$, where $\mathbb P$ is
a set and $\lhd $ a transitive antisymmetric relation on $\mathbb
P$, see \cite[Proposition 3.5.2]{Pierce}.

Let $M$ be a primitive monoid and $p\in \mathbb P (M)$. Then $p$ is
said to be {\it free} in case $p\ntriangleleft p$. Otherwise $p$ is
{\it regular}, see \cite[Section 2]{APW}. In case all the primes of
$M$ are free, the relation $\lhd$ is completely determined by the
poset $(\mathbb P(M),\le )$, where $\le $ is the restriction to
$\mathbb P (M)$ of the algebraic order of $M$. Namely, we have
$q\lhd p$ if and only if $q<p$. In this way we obtain mutually
inverse (up to isomorphism) correspondences $M\mapsto (\mathbb
P(M),\le )$ and $(\mathbb P,\le )\mapsto M(\mathbb P)$ between
primitive monoids having all primes free and posets.

We can now describe the properties of the class of $K$-algebras
$Q_K(\mathbb P)$ that we associate with finite posets $\mathbb P$.
The main result of this paper is:

\begin{theor*}
Let $M$ be a finitely generated primitive monoid such that all
primes of $M$ are free and let $K$ be a field. Let $\mathbb P$ be
the finite set of primes of $M$, endowed with the restriction of the
algebraic order on $M$. Let $Q_K(\mathbb P)$ be the $K$-algebra
described explicitly in Definition \ref{defi: Q(M)}. Then
$Q_K(\mathbb P)$ is a von Neumann regular ring and the natural
monoid homomorphism
$$\psi \colon M\rightarrow \mon{Q_K(\mathbb P)}$$ is an isomorphism.
\end{theor*}

This gives a positive solution to the realization problem for the
class of finitely generated primitive monoids with all prime
elements free. The connection of our work with the paper \cite{AB2}
is as follows. In \cite{AB2}, a von Neumann regular algebra $Q_K(E)$
has been attached to every quiver and every field $K$, in such a way
that there is an isomorphism $\mon{Q_K(E)}\cong M(E)$. Here $M(E)$
is a certain conical refinement monoid associated with the quiver
$E$, with generators and relations given explicitly from the
combinatorial structure of $E$. It has been shown in \cite{APW} that
a finitely generated primitive monoid $M$ is isomorphic to a graph
monoid $M(E)$ for some quiver $E$ if and only if every free prime of
$M$ has at most one free lower cover. (Here a lower cover of a prime
$p$ is a prime $q$ such that $q<p$ and $[q,p]=\{q,p\}$.) In
particular the primitive monoid $M$ given by $M=\langle p,a,b:
p=p+a=p+b \rangle$ cannot be realized by an algebra $Q_K(E)$
corresponding to a quiver $E$. The corresponding poset $\mathbb P$
is depicted below. Observe that the (free) prime $p$ has two (free)
lower covers, namely $a$ and $b$.

\begin{figure}[htb]
 \[
 {
 \xymatrix{
 & p \ar@{-}[ld]\ar@{-}[rd] & \\
 a & & b
 }
 }
 \]
\caption{The poset $\mathbb P (M)$ for the monoid $M=\langle
p,a,b\mid p=p+a=p+b \rangle $.} \label{Fig:Graphpabp}
\end{figure}
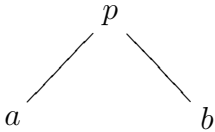

Consequently our construction significantly enlarges the class of
monoids known to be realizable. Moreover, the algebras $Q_K(\mathbb
P)$ have a functorial property with respect to certain morphisms
between posets (Proposition \ref{uniprop}). This is analogous to the
functoriality property of the regular algebra of a quiver (see
\cite[page 234]{AB2}). It is reasonable to expect that a suitable
combination of the methods developed in \cite{AB2} and the methods
developed in the present paper will lead to a general construction
of a von Neumann regular ring realizing every finitely generated
conical refinement monoid. In this direction, it is worth mentioning
that the construction of $Q_K(\mathbb P)$ in Definition \ref{defi:
Q(M)} gives an algebra $Q_L(E)$ associated with a suitable quiver
$E$, where $L=K(t_1,t_2,\dots ,)$, in case every element in $\mathbb
P$ has at most one lower cover. In particular, for the poset with
two elements $q,p$ with $q<p$, we get the Toeplitz algebra
$Q_L(E_1)$ mentioned at the beginning of the introduction. Certainly
the next step in the realization problem consists of realizing all
the finitely generated primitive monoids. A monoid $M$ in this
class, which is not covered by the results in \cite{AB2} or the
results in the present paper, must have both regular and free
primes, and some free prime of $M$ must have at least two free lower
covers. For instance, the monoid $$M=\langle q,p,a,b\mid q=2q=q+p\,
, \, p=p+a=p+b \rangle $$ satisfies these properties.

Although the main construction in the present paper shares some
resemblances with the one in \cite{AB2}---both objects have the form
$\mathcal A \Upsilon ^{-1}$, where $\mathcal A$ is the path algebra
of a quiver in \cite{AB2}, and $\mathcal A$ is an algebra described
by generators and relations coming from the structure of a poset in
the present paper (see Definition \ref{defi: Q(M)}), and $\Upsilon$
is a certain set of morphisms between finitely generated projective
$\mathcal A$-modules---a completely new set of techniques has been
developed to deal with our construction here. These techniques
include the fundamental study (in Section \ref{sect:pullbacks}) of
the conditions for the preservation of pullbacks under the functor
$\mathcal V (-)$, which indeed has dictated the form of the
relations used in \ref{defi: Q(M)} to define the algebra $\mathcal
A$.

We now summarize the contents of the paper. In Section
\ref{sect:preldefs} we review some basic definitions and results on
posets, monoids and rings. Section \ref{QQQ} contains the definition
of the $K$-algebra $Q_K(\mathbb P)$ associated to a finite poset
$\mathbb P$ and analyzes the functorial behaviour and algebraic
properties of this construction. Moreover, we give a Toeplitz-like
representation of this algebra.

Sections \ref{sect:pullbacks} and \ref{sect:pushout} contain
technical results needed to gain control on the relationship between
von Neumann regular rings $R$ and their monoids $\mon{R}$ under
natural categorical operations such as pullbacks and pushouts. These
results are of independent interest, and most likely will play a
role in future developments of the theory.

Section \ref{sect:building} develops a generalization of the
construction in \cite{AB2} for a particular class of quivers, which
will be used in the proof of our main result. This generalized
construction can be studied in a more general setting, see
\cite{AB3}, but we present here a direct approach to the results
which are needed in the present paper.

Finally, Section \ref{sect:theproof} contains the proof of our main
result, which is based upon a reconstruction technique of a finite
poset from the family of its maximal chains. The technical tools
developed in the previous sections enable us to mimic the mentioned
reconstruction at the ring level and at the monoid level.

\smallskip

\section{Preliminary definitions}
\label{sect:preldefs}

All rings in this paper will be associative and all monoids will be
abelian.  A (not necessarily unital) ring $R$ is {\it von Neumann
regular} if for every $a\in R$ there is $b\in R$ such that $a=aba$.
Our basic reference for the theory of von Neumann regular rings is
\cite{vnrr}.

For a (not necessarily unital) ring $R$, let $M_{\infty}(R)$ be the
directed union of $M_n(R)$ ($n\in\mathbb N$), where the transition
maps $M_n(R)\to
M_{n+1}(R)$ are given by $x\mapsto \left( \smallmatrix x&0\\
0&0\endsmallmatrix \right)$. Two idempotents $e,f\in M_{\infty}(R)$
are {\it equivalent} in case there are $x\in eM_{\infty}(R)f$ and
$y\in fM_{\infty}(R)e$ such that $xy=e$ and $yx=f$. We define
$\mon{R}$ to be the set of equivalence classes $\mon{e}$ of
idempotents $e$ in $M_\infty(R)$ with the operation
$$\mon{e}+\mon{f} :=
\mathcal V \bigl(  \left( \smallmatrix e&0\\ 0&f
\endsmallmatrix \right) \bigr)$$
for idempotents $e,f\in M_\infty(R)$. The classes $\mon{e}$ are
often denoted also by $[e]$. For unital $R$, the monoid $\mon{R}$ is
the monoid of isomorphism classes of finitely generated projective
left $R$-modules, where the operation is induced by direct sum. If
$I$ is an ideal of a unital ring $R$, then $\mon{I}$ can be
identified with the monoid of isomorphism classes of finitely
generated projective left $R$-modules $P$ such that $P=IP$.

If $R$ is an exchange ring (in particular, if $R$ is von Neumann
regular), then $\mon{R}$ is a conical refinement monoid, see
\cite[Corollary 1.3]{AGOP}. (This is true even in the non-unital
case by \cite[Proposition 1.5(b)]{Aext}.)

Let $M$ be a monoid. An {\it order-ideal} of~$M$ is a nonempty
subset $I$ of~$M$ such that $x+y\in I$ if{f} $x\in I$ and $y\in
I$, for all $x,y\in M$. In this case, the equivalence relation
$\equiv_I$ defined on~$M$ by the rule
 \[
 x\equiv_Iy\ \Longleftrightarrow\ (\exists u,v\in I)(x+u=y+v),\quad
 \text{for all }x,y\in M
 \]
is a monoid congruence of~$M$. We put $M/I=M/{\equiv_I}$ and we
 shall say that $M/I$ is an
\emph{ideal quotient} of~$M$. We denote by
 \[
 M\mid a = \{ x\in M: (\exists n\in\mathbb Z^+)(x\leq na) \}
 \]
the order-ideal generated by an element $a\in M$. Similarly $M\mid
S$ will denote the order-ideal of $M$ generated by a subset $S$ of
$M$.

When $M$ is a conical refinement monoid, the set $\mathcal L(M)$
of order-ideals of $M$ forms a complete distributive lattice, with
suprema and infima given by the sum and the intersection of
order-ideals respectively.

Let us denote by $\mathcal  L (R)$ the lattice of (two-sided)
ideals of $R$, and by $\mathcal L (M) $ the lattice of
order-ideals of $M$.

\begin{prop} {\rm (cf. \cite[Proposition 1.4]{AGOP})}
\label{prop:wellknownideals} If $R$ is von Neumann regular, then
there is a lattice isomorphism $\mathcal L (R)\to \mathcal L
(\mon{R})$, $I\mapsto \mon{I}$ from $\mathcal L (R)$ onto $\mathcal
L (\mon{R})$. Moreover $\mon{R/I}\cong \mon{R}/\mon{I}$ for any
ideal $I$ of $R$.
\end{prop}

Say that a subset $A$ of a poset $\mathbb P$ is a {\it lower
subset} in case $q\le p$ and $p\in A$ imply $q\in A$. Again the
set $\mathcal L (\mathbb P)$ of all lower subsets of $\mathbb P$
forms a complete distributive lattice, which is a sublattice of
the Boolean lattice ${\bf 2}^{\mathbb P}$.

If $M$ is a primitive monoid then the set of primes of $M$,
$\mathbb P (M)$, is a poset with the partial order $\le$ induced
from the algebraic order of $M$, and we easily get:

\begin{prop}
\label{wellknownlattice} For a primitive monoid $M$, there is a
lattice isomorphism
$$\mathcal L (M)\cong \mathcal L (\mathbb P (M)), \qquad S\mapsto
\mathbb P (S)=\mathbb P (M)\cap S.$$
\end{prop}

For an element $p$ of a poset $\mathbb P$, write
$$\rL(p)=\rL(\mathbb P ,p)=\{q\in \mathbb P : q<p \text{ and } [q,p]=\{q,p\}\}.$$
An element of $\rL (p)$ is called a {\it lower cover} of $p$.

For any prime element~$p$ in a refinement monoid~$M$, the map
 \[
 \phi_p\colon M\to\mathbb Z^{\infty},\quad x\mapsto\sup\famm{n\in\mathbb Z^+}{np\leq x}
 \]
is a monoid homomorphism from~$M$ to~$\mathbb Z^{\infty}:=\mathbb
Z ^+\cup \{\infty\}$, see \cite[Theorem~5.4]{Brook01}.
Furthermore, if~$M$ is primitive, then the map
 \begin{equation}\label{Eq:Defphi}
 \phi\colon M\to(\mathbb Z^{\infty})^{\mathbb P(M)},\quad x\mapsto\famm{\phi_p(x)}{p\in\mathbb P(M)}
 \end{equation}
is a monoid embedding as well as an $\lhd$-embedding, see
\cite[Theorem~5.11]{Brook01} or \cite[Corollary~6.14]{WDim}.

A monoid $M$ is said to be {\it separative} in case, whenever
$a,b\in M$ and $a+a=a+b=b+b$, then we have $a=b$. Similarly $M$ is
{\it strongly separative} in case $a+a=a+b$ implies $a=b$ for
$a,b\in M$. A ring $R$ is said to be {\it (strongly) separative} in
case $\mon{R}$ is (strongly) separative, see \cite{AGOP} for
background and various equivalent conditions. Every primely
generated refinement monoid is separative \cite[Theorem
4.5]{Brook01}. In particular every primitive monoid is separative.
Moreover, a primitive monoid $M$ is strongly separative if and only
if all the primes in $M$ are free, see \cite[Theorem 4.5, Corollary
5.9]{Brook01}. Thus, the class of monoids that we will realize in
this paper (as monoids of projectives over regular rings) coincides
exactly with the strongly separative primitive monoids.

\section{The algebras $Q_K(\mathbb P)$}
\label{QQQ}

Recall from the introduction that (finitely generated) primitive
monoids $M$ with all primes free are determined by the (finite)
posets $\mathbb P (M)$ of their prime elements. The construction
below has functorial properties with respect to some maps of posets,
so it is better thought of as a functor from finite posets to
$K$-algebras. However we will use both notations $Q_K(M)$ and
$Q_K(\mathbb P (M))$ interchangeably.

\begin{defi}
\label{defi: Q(M)} Let $M$ be a finitely generated primitive monoid
with all primes free, and let $\mathbb P$ be its finite poset of
primes. Fix a field $K$. For $p\in \mathbb P$ denote
$n_p:=|\rL(\mathbb P,p)|$.

Let $L=K(t_1,t_2,\dots ,)$ be an infinite purely transcendental
extension of $K$. Let $\mathcal A_0$ be the semisimple $L$-algebra
generated by a family of orthogonal idempotents $\{e(p): p\in
\mathbb P\}$ with sum $1$, and, for each $p\in \mathbb P$ with
$n_p>0$, a family of orthogonal idempotents $\{e(p,q):q\in \rL(p)\}$
such that $e(p)e(p,q)=e(p,q)=e(p,q)e(p)$.

For convenience, put $e'(p):=e(p)-\sum _{q\in \rL(p)}e(p,q)$, with
$e'(p)=e(p)$ in case $n_p=0$. Then we have
\begin{equation}
\label{eq:A.1} \mathcal A _0:=\prod _{p\in \mathbb P}e'(p)L\times
\prod_{p\in \mathbb P}\prod _{q\in \rL(p)}e(p,q)L
\end{equation}
and we also have an orthogonal decomposition
\begin{equation}
\label{eq:A.2} e(p)=e'(p)+\sum _{q\in \rL(p)}e(p,q)\qquad (p\in
\mathbb P).
\end{equation}

Let $\mathcal A _1$ be the $L$-algebra generated by $\mathcal A _0$
and a family $\{ \alpha _{p,q}: p\in \mathbb P , q\in \rL(p)\}$,
subject to the following relations:

\begin{equation}
\label{eq:A.3} \alpha _{p,q}e(p)=\alpha _{p,q}=(e(p)-e(p,q))\alpha
_{p,q} \qquad (p\in \mathbb P, q\in \rL(p)),
\end{equation}
\begin{equation}
\label{eq:A.4} \alpha _{p,q'}e(p,q)=e(p,q)\alpha _{p,q'}\qquad
\text{for } q\ne q'.
\end{equation}
\begin{equation}
\label{eq:A.5} \alpha _{p,q'}\alpha_{p,q}=\alpha _{p,q}\alpha
_{p,q'}\qquad  (q, q'\in \rL(p)),
\end{equation}

For each polynomial $f(x_{q})\in L[x_q:q\in \rL(p)]$ in commuting
variables $\{x_q:q\in \rL(p)\}$ and each $q'\in \rL(p)$, write
$v_{q'}(f)$ for the valuation of $f(x_q)$, seen as a polynomial in
the one-variable polynomial ring $(L[x_q:q\ne q'])[x_{q'}]$, at the
ideal generated by $x_{q'}$. In other words, $v_{q'}(f)$ is the
highest integer $n$ such that $x_{q'}^n$ divides $f$. Write
$$v(f)=\text{max}\{v_{q'}(f):q'\in \rL(p)\}.$$
Let $\Sigma(p)$ be the set of all elements of $e(p)\mathcal
A_1e(p)$ given by

\begin{equation}
\label{eq:A.6} \Sigma (p)=\{f(\alpha_{p,q}): v(f)=0\}.
\end{equation}

Set

\begin{equation}
\label{eq:A.7} \Sigma :=\bigcup _{p\in\mathbb P} \Sigma (p).
\end{equation}

 Let $\mathcal T (\mathbb P)$ be the quiver having as vertices the
elements of $\mathbb P$ and having one arrow from $p$ to $q$ if and
only if $q\in \rL(p)$. We assume our posets are endowed with
bijective maps $\{1,\dots, n_p\}\to s_{\mathcal T(\mathbb
P)}^{-1}(p)$ for every $p\in \mathbb P$ such that $n_p>0$. For
convenience, we will refer to the enhanced structure as a labelled
poset.

For $p\in \mathbb P$ with $n_p>0$, we denote by $(p,q_i)$ the image
of $i\in \{1,\dots ,n_p\}$ under the given map $\{1,\dots, n_p\}\to
s_{\mathcal T(\mathbb P)}^{-1}(p)$; we also denote by
$\sigma^p\colon L\to L$ the $K$-algebra  endomorphism determined by
the rule $\sigma^p (t_i)=t_{i+n_p-1} $, and by $$\sigma _j\colon
\{1,\dots ,n_p\}\longrightarrow \{1,\dots ,n_p-1\}$$  the surjective
non-decreasing map sending $j$ and $j+1$ to $j$, for $j<n_p$, and
with $\sigma _{n_p}:=\sigma _{n_p-1}$.

We consider the $K$-algebra $\mathcal A $ obtained by adjoining to
$\mathcal A_1$ a new family of generators $\{ \beta _{p,q}: p\in
\mathbb P , q\in \rL(p)\}$, subject to the relations:

\begin{equation}
\label{eq:A.8} \lambda \beta_{p,q}=\beta_{p,q}\sigma ^p(\lambda )
\qquad (\lambda \in L, q\in \rL(p)),
\end{equation}
\begin{equation}
\label{eq:A.9} \alpha _{p,q_{\ell}}\beta _{p,q_j}=\beta
_{p,q_j}t_{\sigma_j (\ell)},\qquad  (j,\ell\in \{1,\dots ,n_p\},
j\ne \ell)
\end{equation}
and
\begin{equation}
\label{eq:A.10} e(p,q)\beta _{p,q}=\beta_{p,q}=\beta
_{p,q}e(q)\qquad (q\in \rL(p)).
\end{equation}

 Next consider the set of homomorphisms of finitely generated
projective left $\mathcal A$-modules $\Sigma _1=\{\mu _{p,q}:p\in
\mathbb P,q\in \rL(p)\}$, where for $q\in \rL(p)$, the map
$\mu_{p,q}$ is defined as follows:

\begin{equation}
\label{eq:A.11} \mu _{p,q}\colon \mathcal A e(p)\longrightarrow
\mathcal A e(p)\oplus \mathcal A e(q), \quad \mu _{p,q}(r)=(r\alpha
_{p,q},r\beta_{p,q}).
\end{equation}

Finally define the $K$-algebra $Q_K(\mathbb P)=Q_K(M)$ associated
with $\mathbb P$ (or equivalently, with $M$) as

\begin{equation}
\label{eq:A.12} Q_K(\mathbb P)=Q_K(M)=\mathcal A(\Sigma \cup \Sigma
_1)^{-1},
\end{equation}

\noindent where the elements in $\Sigma (p)$ are seen as left
$\mathcal A$-module morphisms $\mathcal Ae(p)\to \mathcal Ae(p)$.
The algebra $Q_K(\mathbb P)$ is a unital $K$-algebra, with unit
$1=\sum _{p\in \mathbb P}e(p)$.\qed
\end{defi}

Let us write explicitly what is the meaning of making the maps in
$\Sigma _1$ invertible in terms of generators and relations. It
amounts to add to $\mathcal A \Sigma ^{-1}$ a family of generators
$\ol{\beta }_{p,q},\ol{\alpha}_{p,q}$ for $q\in \rL(p), p\in\mathbb
P$ with the following relations:

\begin{equation}
\label{eq:A.13} e(p)\ol{\alpha}_{p,q}
=\ol{\alpha}_{p,q}=\ol{\alpha}_{p,q}(e(p)-e(p,q)),\,\,\quad
\ol{\alpha}_{p,q}\alpha_{p,q}  =e(p), \,\,\quad \alpha
_{p,q}\ol{\alpha}_{p,q} =e(p)-e(p,q),
\end{equation}
\begin{equation}
\label{eq:A.14} e(q)\ol{\beta}_{p,q}=\ol{\beta
}_{p,q}=\ol{\beta}_{p,q}e(p,q),\qquad \ol{\beta
}_{p,q}\beta_{p,q}=e(q), \qquad \beta_{p,q}\ol{\beta}_{p,q} =e(p,q).
\end{equation}

\medskip

It is a simple matter to check that the relations above imply the
following ones:

\begin{equation}
\label{eq:A.15} e(p,q)\ol{\alpha} _{p,q'}=\ol{\alpha}
_{p,q'}e(p,q)\quad \text{for } q\ne q';\qquad \ol{\alpha}
_{p,q'}\ol{\alpha}_{p,q}=\ol{\alpha} _{p,q}\ol{\alpha} _{p,q'}\quad
(q, q'\in \rL(p)),
\end{equation}
\begin{equation}
\label{eq:A.16} \ol{\beta }_{p,q}\lambda
=\sigma^p(\lambda)\ol{\beta}_{p,q} \qquad (\lambda \in L, q\in
\rL(p)).
\end{equation}
\begin{equation}
\label{eq:A.17} \ol{\beta}_{p,q_j}\alpha _{p,q_{\ell}}=t_{\sigma_j
(\ell)}\ol{\beta}_{p,q_j}, \qquad (j,\ell\in \{1,\dots ,n_p\}, j\ne
\ell)
\end{equation}
\begin{equation}
\label{eq:A.18} t_{\sigma_j
(\ell)}^{-1}\ol{\beta}_{p,q_j}=\ol{\beta}_{p,q_j}\ol{\alpha}_{p,q_{\ell}},
\qquad (j,\ell\in \{1,\dots ,n_p\}, j\ne \ell)
\end{equation}

\begin{remark}
Observe that the above relations (\ref{eq:A.13}), (\ref{eq:A.14}),
(\ref{eq:A.15}), (\ref{eq:A.16}), (\ref{eq:A.18}) give that the
$K$-algebra $\mathcal A\Sigma _1^{-1}$ admits a natural involution
$x\mapsto \ol{x}$ which is the identity on $\prod _{p\in \mathbb
P}e'(p)K\times \prod_{p\in \mathbb P}\prod _{q\in \rL(p)}e(p,q)K$,
and sends $t_k$ to $t_k^{-1}$, $\alpha _{p,q}$ to
$\ol{\alpha}_{p,q}$, and $\beta _{p,q}$ to $\ol{\beta}_{p,q}$. The
algebra $\mathcal A\Sigma_1^{-1} $ is analogous to the Leavitt path
algebra of \cite{AA1} (see also \cite{AMP}).
\end{remark}

Since $Q_K(M)e(p)\cong Q_K(M)e(p)\oplus Q_K(M)e(q)$ for every $q\in
\rL(p)$, there is a unique monoid homomorphism $\psi \colon
M\rightarrow \mon{Q_K(M)}$ such that $\psi (p)=[e(p)]$.
 Our main result can now be stated as follows:

\begin{theor}
\label{main} Let $M$ be a finitely generated primitive monoid such
that all primes of $M$ are free and let $K$ be a field. Then
$Q_K(M)=Q_K(\mathbb P(M))$ is a von Neumann regular ring and the
natural monoid homomorphism
$$\psi \colon M\rightarrow \mon{Q_K(M)}$$ is an isomorphism.
\end{theor}

This will be proven in Section \ref{sect:theproof}.

\begin{remark}
\label{rem:precission}
\begin{enumerate}
\item We have considered the structure of a labelled poset in order
to define relation (\ref{eq:A.9}), so that strictly speaking we have
defined an algebra $Q_K(\mathbb P)$ for each {\it labelled poset}
$\mathbb P$. In order to consider the assignment $\mathbb P\mapsto
Q_K(\mathbb P)$ as a functor, this ingredient is needed. However the
properties of $Q_K(\mathbb P)$ do not depend on the chosen maps.
\item In case ~$\mathbb P$ is a \emph{forest}, i.e. $\mathbb P\dnw p:=\setm{q\in \mathbb  P}
{q\leq p}$ is a chain for any~$p\in \mathbb P$, we get that
$Q_K(\mathbb P)$ coincides with $Q_L(E)$, the regular algebra of $E$
over the field $L$, see Section \ref{sect:building} and \cite{AB2},
where $E$ is the quiver obtained by adding to $\mathcal T (\mathbb
P)$ a loop at each vertex $p\in \mathbb P$ such that $n_p>0$. For
instance the quiver $E$ corresponding to the forest $\mathbb
P=\{q,p\}$, where $q<p$, is the quiver $E_1$ of Figure
\ref{Fig:QuiverS1}. The Leavitt path algebra of $E$ over the field
$L$, see \cite{AA1}, is precisely the algebra $\mathcal A\Sigma_1
^{-1}$ in the notation of Definition \ref{defi: Q(M)}.
\end{enumerate}
\end{remark}

\begin{exem}
\label{exam:main} We consider our motivating example $M=\langle
p,a,b \mid p=p+a=p+b\rangle$, see Figure \ref{Fig:Graphpabp}.
Observe that $\mathbb P: =\mathbb P (M)$ is the poset $\{ p,a,b \}$
with $a<p$ and $b<p$, so that $n_p=2$ and $n_a=n_b=0$, so that
$\mathbb P$ is {\it not} a forest, and the monoid $M$ is {\it not} a
graph monoid (\cite[Theorem 5.1]{APW}). We choose the map
$\{1,2\}\to \rL(p)=\{a,b\}$ given by $1\mapsto a$ and $2\mapsto b$.
The regular algebra $Q_K(\mathbb P)$ associated with the poset
$\mathbb P$ can be described in terms of the regular algebras
$Q_1:=Q_K(S_1)$ and $Q_2:=Q_K(S_2)$ associated with the forests
$S_1$ and $S_2$, where $S_1$ and $S_2$ are the chains $\{a,p\}$ and
$\{b,p\}$ respectively. Since this is the main point of the whole
construction, we are going to provide some hints on the general
strategy to prove Theorem \ref{main} in this particularly simple
case.

Note that, by Remark \ref{rem:precission}(2), the algebras $Q_1$ and
$Q_2$ are isomorphic to the regular $L$-algebra $Q_L(E_1)$ attached
to the quiver $E_1$ of Figure \ref{Fig:QuiverS1}. Note that $Q_1$
has unit element $1_{Q_1}=e(p)+e(a)$, and has generators $\alpha
_1,\ol{\alpha}_1,\beta _1,\ol{\beta}_1$, with
$$\ol{\alpha}_1\alpha_1=e(p)=\alpha _1\ol{\alpha}_1+\beta
_1\ol{\beta }_1$$ and $\ol{\beta}_1\beta_1=e(a)$. Moreover, all
elements of the form $e(p)f(\alpha_1)$, where $f\in L[z]$ is a
polynomial with nonzero constant term, are in $\Sigma (p)$ and thus
the elements $e(p)f(\alpha_1)$ are invertible in $e(p)Q_1e(p)$. Let
$I_1$ be the ideal of $Q_1$ generated by $e(a)$. It follows from the
above that $Q_1/I_1\cong L(z)$, the rational function field in one
variable $z$, where $z$ corresponds to the class of $\alpha _1$ in
$Q_1/I_1$. Since $L(z)\cong L$, we may consider a surjective
homomorphism  $\pi _1\colon Q_1\to L$ such that $\pi
_1(\alpha_1)=t_1$, and $\pi _1(t_i)=t_{i+1}$ for all $i$. Note that
the kernel of $\pi _1$ is $I_1$. Similarly, we have a surjective
homomorphism $\pi _2\colon Q_2\to L$, with kernel $I_2:=Q_2e(b)Q_2$
sending $\alpha _2$ to $t_2$, $t_1$ to $t_1$ and $t_i$ to $t_{i+1}$
for all $i>1$, where $\alpha_2,\ol{\alpha}_2,\beta _2,\ol{\beta}_2 $
are canonical generators for $Q_2$. It is a simple exercise to show
that there is a pullback diagram:
\begin{equation}
\label{diagram-example}
\begin{CD}
Q_K(\mathbb P) @>\rho _1>> Q_1\\
   @V\rho_2 VV  @VV\pi _1V\\
Q_2 @>\pi_2>> L  ,
\end{CD}
\end{equation}
for suitably defined $K$-algebra homomorphisms $\rho _i\colon
Q_K(\mathbb P)\to Q_i$. (The key here is {\it to define}
$\rho_1(\alpha_{p,b})=e(p)t_1$, and $\rho _1(e(p)t_i)=e(p)t_{i+1}$,
$\rho _1(e(a)t_i)=e(a)t_i$ for all $i$, and similarly with $\rho
_2$, so that the defining relations of $Q_K(\mathbb P)$ are
preserved by $\rho _i$, and $\pi _1\circ \rho_1 =\pi _2\circ \rho
_2$.)

The results in Section \ref{sect:pullbacks} guarantee that the
functor $\mathcal V$ sends the pullback diagram
(\ref{diagram-example}) to a pullback diagram of monoids (see
Example \ref{exam:pullbackchecked}), and then the results in
\cite{AB2} applied to the quiver algebras $Q_i$ enable us to
conclude that $\mon{Q_K(\mathbb P)}\cong M$, as desired.\qed
\end{exem}

The universal property of the construction is as follows. Call a
morphism of posets $f\colon \mathbb P\to \mathbb P '$ a {\it
complete homomorphism} in case it is injective and has the property
that for any $p\in \mathbb P$ such that $\rL(\mathbb P,p)\ne
\emptyset$ we have that $f$ restricts to a bijection from
$\rL(\mathbb P,p)$ onto $\rL(\mathbb P', f(p))$. Observe that this
is the same as a complete graph homomorphism $\mathcal T(\mathbb
P)\to \mathcal T(\mathbb P ')$ between the quivers introduced in
Definition \ref{defi: Q(M)}, in the sense of \cite[Section 3]{AMP}.
If $\mathbb P$ and $\mathbb P'$ are labelled posets, we define a
{\it complete homomorphism} as a complete graph homomorphism
$f\colon \mathcal T(\mathbb P)\to \mathcal T(\mathbb P ')$
respecting the labels.

\medskip

The {\it category of non-unital $K$-algebras} has as objects all not
necessarily unital $K$-algebras, and as morphisms all the
homomorphisms of $K$-algebras. In the next proposition, note that
the algebras $Q_K(\mathbb P)$ are unital, but the homomorphisms
appearing there are not unital except in the trivial cases.

\begin{prop}
\label{uniprop} The map $\mathbb P\mapsto Q_K(\mathbb P)$ extends to
a functor from the category of finite labelled posets and complete
homomorphisms to the category of non-unital $K$-algebras.
\end{prop}

\begin{proof}
Straightforward. Observe that, for a complete homomorphism
$f\colon \mathbb P \to \mathbb P'$, the image of $1_{Q_K(\mathbb
P)}$ is the idempotent $\sum_{p\in\mathbb P} e(f(p))$.
\end{proof}

\begin{remark}
\label{dirlimits} Using the functoriality established in Proposition
\ref{uniprop}, one can generalize Theorem \ref{main} to some
infinite posets by taking direct limits, see for instance the proofs
of \cite[Theorem 3.5]{AMP} and \cite[Theorem 4.4]{AB2}. Note that
the class of (labelled) posets that can be obtained as a direct
limit of a directed system of finite labelled posets, with complete
homomorphisms as transition maps, includes the class of lower finite
posets, where a poset~$\mathbb P$ is \emph{lower finite} in case the
subset $\mathbb P\dnw p=\setm{q\in \mathbb P}{q\leq p}$ is finite,
for any~$p\in P$. A characterization of graph monoids among
primitive monoids $M$ such that $\mathbb P (M)$ is lower finite has
been obtained in \cite[Theorem 5.1]{APW}.
\end{remark}

We start analyzing the basic algebraic properties of $Q_K(\mathbb
P)$. We begin with a general observation, stating that universal
localization commutes with adjoining generators and relations to a
given algebra. For a field $K$, the coproduct of two unital
$K$-algebras $\mathcal A$ and $\mathcal B$ will be denoted by
$\mathcal A
*_K\mathcal B$; see \cite{BergCo} for the definition and properties of such
coproducts.

\begin{prop}
\label{commadjoin} Let $\mathcal A$ be a unital $K$-algebra over a
field $K$, and let $\Sigma$ be a set of morphisms between finitely
generated projective $\mathcal A$-modules. Let $\mathcal B$ be
another unital $K$-algebra, and let $I$ be an ideal of the coproduct
$\mathcal A *_K\mathcal B$. Set $\,\, \mathcal C=(\mathcal
A*_K\mathcal B)/I$. Let $\overline{I}$ be the ideal of $(\mathcal
A\Sigma ^{-1})*_K\mathcal B$ generated by the image of $I$ under the
canonical map $\mathcal A*_K \mathcal B\to (\mathcal A\Sigma
^{-1})*_K\mathcal B$, and let $\overline{\Sigma} $ be the image of
$\Sigma$ under the natural map $\mathcal A\to \mathcal C$ (obtained
by composing the natural map $\mathcal A\to \mathcal A*_K \mathcal
B$ with the canonical projection $\mathcal A*_K \mathcal B\to
\mathcal C$). It follows that there is a natural isomorphism
$\mathcal C\overline{\Sigma}^{-1}\cong (\mathcal A\Sigma
^{-1}*_K\mathcal B)/\ol{I}$.
\end{prop}

\begin{proof}
First we show that $\mathcal A \Sigma ^{-1}*_K\mathcal B\cong
(\mathcal A*_K\mathcal B)\widetilde{\Sigma}^{-1}$ (over $\mathcal
A *_K \mathcal B$), where $\widetilde{\Sigma}$ is the image of
$\Sigma$ under the map $\mathcal A\to \mathcal A*_K\mathcal B$.
This follows from the easily proved fact that both algebras
satisfy an obvious universal condition with respect to $K$-algebra
homomorphisms $\psi _1\colon \mathcal A\to S$ and $\psi _2\colon
\mathcal B\to S$, with $\psi_1$ being $\Sigma$-inverting.

Let $\widetilde{I}$ be the ideal of $(\mathcal A*_K\mathcal
B)\widetilde{\Sigma}^{-1}$ generated by the image of $I$. By
\cite[Proposition 4.3]{DPP}, we have $((\mathcal A*_K\mathcal
B)\widetilde{\Sigma}^{-1})/\widetilde{I}\cong \mathcal
C\overline{\Sigma}^{-1}$ and thus combining with the above
isomorphism we get
$$\mathcal C\overline{\Sigma}^{-1}\cong ((\mathcal A*_K\mathcal
B)\widetilde{\Sigma}^{-1})/\widetilde{I}\cong (\mathcal A \Sigma
^{-1}*_K\mathcal B)/\ol{I},$$ as desired.
\end{proof}

\begin{lem}
\label{commuting} For $p\in \mathbb P$, $q\in \rL(p)$ and $f\in
\Sigma (p)$ we have
\begin{equation}
\label{eq:A.19}
e(p,q)f^{-1}=(f_0')^{-1}\ol{w}e(p,q)=e(p,q)(f_0')^{-1}\ol{w},
\end{equation}
and
\begin{equation}
\label{eq:A.20}
\ol{\alpha}_{p,q}f^{-1}=f^{-1}\ol{\alpha}_{p,q}+f^{-1}(f_0')^{-1}g\ol{w}e(p,q),
\end{equation}
 where $w$ is a monomial in $\{\alpha_{p,q'}\mid q'\ne q\}$, and
$f_0'\in L[\alpha _{p,q'}:q'\ne q]\cap \Sigma (p)$, and $g\in
L[\alpha _{p,q'}]$.
\end{lem}

\begin{proof}
Write
$$f=f_0+\alpha _{p,q}f_1+\cdots +\alpha _{p,q}^mf_m,$$
where $f_s\in L[\alpha _{pq'}:q'\ne q]$. Then $e(p,q)f=f_0e(p,q)$ by
relations (\ref{eq:A.3}) and (\ref{eq:A.4}). Now write
$$f_0=wf_0'$$
with $w$ a monomial in $\{\alpha _{p,q'}:q'\ne q\}$ and $f_0'\in
\Sigma (p)\cap L[\alpha _{p,q'}:q'\ne q]$. Then we get that $f_0$ is
left invertible in $e(p)Q_K(\mathbb P)e(p)$, because
$(f_0')^{-1}\ol{w}f_0=e(p)$. Thus $$e(p,q)f^{-1}
=(f_0')^{-1}\ol{w}e(p,q),$$ as desired. Also
$(f_0')^{-1}\ol{w}e(p,q)=e(p,q)(f_0')^{-1}\ol{w}$ because of
(\ref{eq:A.4}) and (\ref{eq:A.15}).

Now, observe that, since
$\ol{\alpha}_{p,q}f_i=f_i\ol{\alpha}_{p,q}$ for all $i$, we have
$$\ol{\alpha}_{p,q}f-f\ol{\alpha}_{p,q}=(f_1+\alpha_{p,q}f_2+\cdots +\alpha_{p,q}^{m-1}f_m)e(p,q)$$
so that, multiplying this relation on the left and on the right by
$f^{-1}$ and using (\ref{eq:A.19}), we get equation (\ref{eq:A.20})
with $g=-f_1-\alpha _{p,q}f_2-\cdots -\alpha _{p,q}^{m-1}f_m$.
\end{proof}

Observe that Proposition \ref{commadjoin} says that $Q_K(\mathbb P)$
can be obtained as follows. First consider the commutative
$L$-algebra
$$\mathcal A _2=\prod _{p\in \mathbb P}e(p)L[\alpha_{p,q}:q\in
\rL(p)]\Sigma (p)^{-1},$$ then we adjoin to $\mathcal A _2$ the
family of orthogonal idempotents $\{e(p,q):q\in \rL(p), p\in \mathbb
P\}$, with $e(p,q)e(p)=e(p,q)=e(p)e(p,q)$, and elements
$\{\ol{\alpha}_{p,q}:p\in \mathbb P, q\in \rL(p)\}$, and impose
relations (\ref{eq:A.3}), (\ref{eq:A.4}) and (\ref{eq:A.13}), to
obtain a new algebra $\mathcal A _3$. Note that
$$\mathcal A _3=\prod_{p\in \mathbb P}e(p)\mathcal A _3e(p).$$
Finally the algebra $Q_K(\mathbb P)$ is obtained by adjoining to
$\mathcal A _3$ generators $\beta _{p,q},\ol{\beta }_{p,q}$ for
$q\in \rL(p)$, subject to relations (\ref{eq:A.8}), (\ref{eq:A.9}),
(\ref{eq:A.10}) and (\ref{eq:A.14}).

\begin{lem}
\label{lem:techs} With the above notation we have

{\rm(1)} $e(p)\mathcal A _3\beta_{p,q}\subseteq \sum _{f\in \Sigma
(p)}\sum _{i=0}^{\infty} f^{-1}\alpha _{p,q}^i\beta_{p,q}L$.

{\rm (2)} $\ol{\beta}_{p,q}\mathcal A_3e(p)\subseteq \sum
_{i=0}^{\infty} L\ol{\beta}_{p,q}\ol{\alpha}_{p,q}^i$.

{\rm (3)} $\ol{\beta}_{p,q}\mathcal A_3\beta_{p,q}\subseteq
Le(q)$.

{\rm (4)} $\ol{\beta}_{p,q}\mathcal A_3\beta_{p,q'}=0$ for $q\ne
q'$.
\end{lem}

\begin{proof}
(1) Write $T_{(p,q)}:=\sum _{f\in \Sigma (p)}\sum _{i=0}^{\infty}
f^{-1}\alpha _{p,q}^i\beta_{p,q}L $. It suffices to check that
$\alpha _{p,q'}T_{(p,q)}\subseteq T_{(p,q)}$ and
$e(p,q')T_{(p,q)}\subseteq T_{(p,q)}$ and
$\ol{\alpha}_{p,q'}T_{(p,q)}\subseteq T_{(p,q)}$. The first relation
follows from (\ref{eq:A.9}). The containment
$$e(p,q')(\sum _{i=0}^{\infty}
\alpha _{p,q}^i\beta_{p,q}L)\subseteq \sum _{i=0}^{\infty} \alpha
_{p,q}^i\beta_{p,q}L$$ follows from (\ref{eq:A.3}), (\ref{eq:A.4})
and (\ref{eq:A.10}). It follows from this containment and from
(\ref{eq:A.19}) that, to show that $e(p,q')T_{(p,q)}\subseteq
T_{(p,q)}$, it suffices to check that
$${\ol{w}}\beta_{p,q}L\subseteq \beta_{p,q}L,$$ where $\ol{w}$ is a monomial
 in $\ol{\alpha}_{p,q''}$, $\, (q''\in \rL(p))$. Observe that relation
(\ref{eq:A.9}) implies that $\ol{\alpha} _{p,q_{\ell}}\beta
_{p,q_j}=\beta _{p,q_j}t_{\sigma_j (\ell)}^{-1}$ for $j\ne \ell$,
and that $\ol{\alpha}_{p,q}\beta _{p,q}=0$. We have therefore shown
that $e(p,q')T_{(p,q)}\subseteq T_{(p,q)}$. Now it follows from
(\ref{eq:A.20}) and the above that
$\ol{\alpha}_{p,q'}T_{(p,q)}\subseteq T_{(p,q)}$, completing the
proof of (1).

(2) This follows by ``conjugating" the above relations and the
following identity, for $f\in \Sigma (p)$, which comes from
(\ref{eq:A.19}) and (\ref{eq:A.16}), (\ref{eq:A.17}),
(\ref{eq:A.18}):
$$\ol{\beta}_{p,q_j}f^{-1}=\ol{\beta}_{p,q_j}e(p,q_j)f^{-1}=\ol{\beta}_{p,q_j}(f_0')^{-1}\ol{w}e(p,q_j)=
(\sigma ^p(f_0')(t_{\sigma
_j(\ell)}))^{-1}w(t_{\sigma_j(\ell)}^{-1})\ol{\beta}_{p,q_j},$$
where $w$ is a monomial in $\{\alpha_{p,q_{\ell}}: \ell\ne j\}$,
and $f_0'\in L[\alpha _{p,q_{\ell}}:\ell\ne j]\cap \Sigma (p)$,
and $\sigma ^p(f_0')(t_{\sigma_j(\ell)})\in L$ is the polynomial
obtained by applying $\sigma ^p$ to all the coefficients of $f_0'$
and replacing $\alpha _{p,q_{\ell}}$ with $t_{\sigma_j(\ell)}$,
for $\ell \ne j$.

(3) and (4) follow from (2) and relations (\ref{eq:A.14}),
(\ref{eq:A.9}).
\end{proof}

Let $\got m =z_1^{a_1}\cdots z_k^{a_k}$ be a commutative monomial in
$z_1^{\pm},\dots ,z_k^{\pm}$, so that $a_i\in \mathbb Z$. Assume that $p\in
\mathbb P$ and that $\rL(p)=\{q_1,\dots ,q_k\}$. Then we denote by $
\got m (\alpha _{p,q_1},\dots ,\alpha _{p,q_k})$ the element of
$Q_K(\mathbb P)$ given by formally substituting $z_i^{a_i}$ by
$\alpha _{p,q_i}^{a_i}$ if $a_i>0$, by
$(\ol{\alpha}_{p,q_i})^{-a_i}$ if $a_i<0$ and by $e(p)$ if $a_i=0$.
We say that $ \got m (\alpha _{p,q_1},\dots ,\alpha _{p,q_k})$ is a
{\it monomial in} $\alpha _{p,q_1},\dots
,\alpha_{p,q_k},\ol{\alpha}_{p,q_1},\dots ,\ol{\alpha}_{p,q_k}$.
Observe that it does not depend on the way we order the elements
$q_1,\dots ,q_k$.

Let $\mathcal T :=\mathcal T (\mathbb P)$ be the quiver associated
to the poset $\mathbb P$, see Definition \ref{defi: Q(M)}. A path in
$\mathcal T (\mathbb P)$ is a sequence $(p_1,\dots ,p_u)$ of
elements in $\mathbb P$ such that $p_{i+1}\in \rL(p_i)$ for all $i$.

\begin{lem}
\label{normalform} Every element in $Q_K(\mathbb P)$ can be
written as a finite sum of elements of the form:
\begin{align*}
\tag{*} (f_1)^{-1}\alpha _{p_1,p_2}^{m_1}\beta_{p_1,p_2} & \cdots
(f_{u-1})^{-1}\alpha _{p_{u-1},p_u}^{m_{u-1}}\beta _{p_{u-1},p_u}
 \got m (\alpha_{p_u,q_1},\dots ,\alpha _{p_u,q_k})(f_{u})^{-1} \\
 & \cdot \ol{\beta }_{p_{u+1},p_{u}}
  \ol{\alpha}_{p_{u+1},p_{u}}^{m_{u}}\ol{\beta}_{p_{u+2},p_{u+1}}
\cdots
\ol{\beta}_{p_{u+v+1},p_{u+v}}\ol{\alpha}_{p_{u+v+1},p_{u+v}}^{m_{u+v}},
\end{align*}
with $f_i\in \Sigma (p_i)$ for $i=1,\dots ,u$, $p_{i+1}\in \rL(p_i)$
for all $i=1,\dots ,u-1$ and $p_{u+i}\in \rL(p_{u+i+1})$ for all
$i=1,\dots ,v$. Moreover $\got m (\alpha_{p_u,q_1},\dots
,\alpha_{p_u,q_k})$ is a monomial in $\alpha _{p_u,q_1},\dots
,\alpha _{p_u,q_k},\ol{\alpha}_{p_u,q_1},\linebreak\dots
,\ol{\alpha}_{p_u,q_k}$, where $\rL(p_u)=\{q_1,\dots ,q_k\}$.
 We have
\begin{equation}
\label{eq:A.22} Q_K(\mathbb P)=\bigoplus _{(\gamma_1,\gamma_2)\in
\mathcal T ^2: r(\gamma_1)=r(\gamma _2)} Q_K(\mathbb
P)_{(\gamma_1,\gamma_2)},
\end{equation} where $Q_K(\mathbb P)_{(\gamma_1,\gamma_2)}$ is the
$L$-vector space generated by all terms (*) with $\gamma
_1=(p_1,\dots, p_u)$ and $\gamma _2=(p_{u+v+1}, \dots ,p_u)$.
\end{lem}

\begin{remark}
\label{rema:commondenominator}  Observe that if we have finitely
many elements in $Q_K(\mathbb P)_{(\gamma _1,\gamma_2)}$, we can
express them as a finite sum of terms of the form (*) where the
sequence $(f_1,\dots ,f_{u-1})$ is the same in all the terms. This
can be done by taking common denominators, starting with the
corresponding polynomials in $\Sigma (p_1)$, and then using
relations (\ref{eq:A.9}) to re-write all the terms with a common
$f_1$. Now take common denominators for the polynomials from $\Sigma
(p_2)$ appearing in the new terms, and re-write again. After $u-1$
steps all the terms have the same sequence $(f_1,\dots ,f_{u-1})$.
We will refer to this process as ``taking common denominators".
\end{remark}

\noindent {\it Proof of Lemma \ref{normalform}}  The fact that
$Q_K(\mathbb P)=\sum _{(\gamma_1,\gamma_2)\in \mathcal T ^2:
r(\gamma_1)=r(\gamma _2)} Q_K(\mathbb P)_{(\gamma_1,\gamma_2)}$
follows from Lemma \ref{lem:techs}. (Observe that the idempotents
$e(p,q)$ can be replaced with $\beta_{p,q}\ol{\beta}_{p,q}$.)

 We have to prove that this is a direct sum.
First we claim that the set $Q_K(\mathbb P)_{(p,p)}$ of $L$-linear
combinations of terms of the form $ \got m (\alpha_{p,q_1},\dots
,\alpha _{p,q_k})f^{-1}$, where $f\in \Sigma (p)$, and $\got m $ is
a monomial in $\alpha _{p,q_1},\dots ,\ol{\alpha}_{p,q_k}$, has
trivial intersection with the ideal $J$ of $Q_K(\mathbb P)$
generated by $e(q_1),\dots ,e(q_k)$. Indeed the canonical projection
$Q_K(\mathbb P)\to Q_K(\mathbb P)/J$ induces an isomorphism
$(e(p)+J)(Q_K(\mathbb P)/J)(e(p)+J)\cong e(p)L(z_1,\dots ,z_k)$. The
field $L(z_1,\dots ,z_k)$ is the directed union of the sets $L_f$ of
elements of the form $(\sum _{i=1}^s \lambda _i\got m_i)f^{-1}$,
where $\lambda _i\in L$, $\got m _i$ are (commutative) monomials in
$z_1^{\pm},\dots ,z_k^{\pm}$ and $f\in L[z_1,\dots ,z_k]$ with
$v(f)=0$. The $L$-linear map $e(p)L(z_1,\dots ,z_k)\to
e(p)Q_K(\mathbb P)e(p)$ given by
$$e(p)\got m f^{-1}
\mapsto  \got m (\alpha _{p,q_1},\dots ,\alpha _{p,q_k})f(\alpha
_{p,q_1},\dots ,\alpha_{p,q_k})^{-1}$$ is well-defined due to the
relations (\ref{eq:A.5}) and $\ol{\alpha}_{p,q_i}\alpha
_{p,q_i}=e(p)$ for all $i$. This map provides a set theoretical
section of the projection $e(p)Q_K(\mathbb P)e(p)\to L(z_1,\dots
,z_k)$. This shows our claim.

Given a nonzero element $x$ in $Q_K(\mathbb P)_{(\gamma,\gamma')}$,
with $\gamma,\gamma '\in \mathcal T(A)$ and $r(\gamma )=r(\gamma
')$, written as a finite sum of elements of the form (*), we can
take common denominators (Remark \ref{rema:commondenominator}) and
assume that every summand of the form (*) has exactly the same
sequence $(f_1,\dots ,f_{u-1})$. Let $M_1$ be the highest power of
$\alpha _{p_1,p_2}$ appearing in $x$, that is, such that
$\ol{\alpha}_{p_1,p_2}^{M_1}f_1 x\ne 0$. Then
$$0\ne \ol{\beta}_{p_1,p_2}\ol{\alpha}_{p_1,p_2}^{M_1}f_1x\in Q_K(\mathbb P)_{(\gamma_2,\gamma ')},$$
where $\gamma _2=(p_2,\dots ,p_u)$.
 Proceeding in this way we see that there is
 $$z_1=\ol{\beta}_{p_{u-1},p_u}\ol{\alpha}_{p_1,p_2}^{M_{u-1}}f_{u-1}
 \cdots \ol{\beta}_{p_1,p_2}\ol{\alpha}_{p_1,p_2}^{M_1}f_1$$
and $$z_2=\alpha _{p_{u+v+1},p_{u+v}}^{M_{u+v}}\beta
_{p_{u+v+1},p_{u+v}}\cdots \alpha _{p_{u+1},p_{u}}^{M_u}\beta
_{p_{u+1},p_{u}}$$ such that $z_1xz_2\in Q_K(\mathbb
P)_{(p_u,p_u)}\setminus \{0\}$.

Consider the following partial order on the set of pairs
$\Gamma:=\{(\gamma,\gamma ' )\in \mathcal T^2: r(\gamma )=r(\gamma
')\}$: say that  $(\gamma _1,\gamma _1')\succeq (\gamma _2,\gamma
_2')$ iff $\gamma _2=\gamma _1\gamma _3$ and $\gamma
_2'=\gamma_1'\gamma _3'$ for some paths $\gamma _3,\gamma _3'$ in
$\mathcal T$. In order to show that the sum in (\ref{eq:A.22}) is a
direct sum, suppose that $\sum _{i=1}^sx_{(\gamma_i,\gamma_i')}=0$,
with $(\gamma_i,\gamma _i')$ distinct elements of $\Gamma$ and
$x_{(\gamma_i,\gamma_i')}\in Q_K(\mathbb P)_{(\gamma _i,\gamma
_i')}\setminus \{0\}$ for all $i$. We can assume that
$(\gamma_1,\gamma _1')$ is maximal with respect to $\succeq$
(amongst the pairs $(\gamma _i,\gamma _i')$).

Set $p=r(\gamma _1)=r(\gamma _1')$. Let $z_1$ and $z_2$ be the
elements of $Q_K(\mathbb P)$ corresponding to $x_{(\gamma_1,\gamma
_1')}$, constructed above, so that $z_1x_{(\gamma _1,\gamma
_1')}z_2\in Q_K(\mathbb P)_{(p,p)}\setminus \{0\}$. Observe that, by
maximality of $(\gamma _1,\gamma_1')$, we have that $z_1Q_K(\mathbb
P)_{(\gamma _i,\gamma _i')}z_2=0$ unless $(\gamma _i,\gamma
_i')\preceq (\gamma _1,\gamma _1')$. We conclude that
$$\sum _{i=2}^s z_1Q_K(\mathbb P)_{(\gamma _i,\gamma _i')}z_2\subseteq J,$$
where $J$ is the ideal of $Q_K(\mathbb P)$ generated by
$e(q_1),\dots ,e(q_k)$, with $\rL(p)=\{q_1,\dots ,q_k\}$. We get
$$0\ne z_1x_{(\gamma _1,\gamma_1')}z_2=
-\sum _{i=2}^s z_1x_{(\gamma _i,\gamma _i')}z_2\in Q_K(\mathbb P)_{(p,p)}\cap J=0,$$
which is a contradiction. \qed

\medskip

The method of proof of the above proposition gives the following
nice technical lemma, which will be used to prove injectivity of
some maps defined from $Q_K(\mathbb P)$ to other $K$-algebras. For
$x\in Q_K(\mathbb P)$, write $x=\sum x_{(\gamma _1,\gamma _2)}\in
\bigoplus _{r(\gamma _1)=r(\gamma _2)} Q_K(\mathbb P)_{(\gamma
_1,\gamma_2)}$, see Lemma \ref{normalform}. Then the {\it support}
of $x$ is the set of pairs $(\gamma _1,\gamma _2)$ such that
$x_{(\gamma_1,\gamma_2)}\ne 0$.

\begin{lem}
\label{lem:injectivity} Assume that $x$ is a nonzero element in
$Q_K(M(\mathbb P))$. Then there exist $p\in \mathbb P$ and
$z_1,z_2\in Q_K(\mathbb P)$ such that $z_1xz_2$ has the trivial pair
of paths $(p,p)$ in the support, and all the other elements in the
support of $z_1xz_2$ are of the form $(\gamma _1,\gamma _2)$, where
$\gamma _1$ and $\gamma _2$ are paths in $\mathcal T$ starting in
$p$ and ending in a common vertex.
\end{lem}

\begin{proof}
Consider the partial order $\preceq$ of the proof of Proposition
\ref{normalform} on the support of $x$. Let $(\gamma _1,\gamma _2)$
be a maximal element of the support of $x$ with respect to this
partial order, and set $p=r(\gamma _1)=r(\gamma _2)$. Let $z_1$ and
$z_2$ be the elements of $Q_K(\mathbb P)$ corresponding to
$x_{(\gamma_1,\gamma _2)}$, constructed as in the proof of Lemma
\ref{normalform}. Then $z_1xz_2$ has the desired properties.
\end{proof}

In order to show the coherence of our construction (in particular
to show that $e(p)\ne 0$ for every $p\in \mathbb P$), we will
introduce a canonical faithful representation of the algebras
$Q_K(\mathbb P)$. Observe that for the simple case where $\mathbb
P =\{p_0,p_1\}$ with $p_0<p_1$ our representation of $Q_K(\mathbb
P)$ is basically the canonical representation as Toeplitz
operators on $L[z]$.

For a lower subset $A$ of $\mathbb P $, we will consider the
idempotent $e(A)=\sum _{q\in A}e(q)\in Q_K(\mathbb P)$.

\begin{theor}
\label{theor:repres} The algebra $Q_K(\mathbb P)$ acts faithfully
as $K$-linear maps on a $L$-vector space
$$V(\mathbb P)=\bigoplus _{p\in \mathbb P} V_{\mathbb P}(p),$$
where $e(p)$ is the identity on $V(p)$ and $0$ on $V(q)$ for $q\ne
p$. If $A$ is a lower subset of $\mathbb P$ then we have a canonical
isomorphism $\psi_A\colon Q_K(A)\longrightarrow e(A)Q_K(\mathbb
P)e(A)$, and $V(A)=\oplus _{p\in A}V_{\mathbb P}(p)$. Moreover
$(v)\psi _A(x)=(v)x$ for all $x\in Q_K(A)$ and all $v\in V(A)$.
\end{theor}

\begin{proof}
We will define a {\it right} action of $Q_K(M)$ as $K$-linear
endomorphisms on a $L$-vector space.

We proceed to build the $L$-vector spaces by induction. Set $\mathbb
P^0:=\text{Min}(\mathbb P)$, the set of minimal elements of $\mathbb
P$. If $\mathbb P^0, \dots , \mathbb P^i$ have been defined, set
$\mathbb P^{i+1}=\text{Min}(\mathbb P\setminus \cup _{j=0}^i\mathbb
P^j )$. Obviously $\mathbb P=\cup _{i=0}^r\mathbb P ^i$ for some
$r$, and $\rL(p)\subseteq \cup _{j=0}^i\mathbb P^j$ for $p\in\mathbb
P^{i+1}$. Note that the sets $\cup _{j=0}^i\mathbb P ^j$ are lower
subsets of $\mathbb P$.

For $p\in \mathbb P ^0$, set $V(p)=L$. The action of
$e(p)Q_K(\mathbb P^0)e(p)=L$ on $V(p)=L$ is given by
multiplication.

Now assume that the $L$-vector spaces $V(p)$ have been constructed
for all $p\in \mathbb P^j$, $0\le j\le i$, and that a representation
$\tau _i\colon Q_K(\cup _{j=0}^i \mathbb P ^j)\to \oplus _{p\in \cup
_{j=0}^i \mathbb P^j}V(p)$ satisfying the required conditions has
been defined. Let $p\in \mathbb P ^{i+1}$ and write
$\rL(p)=\{q_1,\dots ,q_k\}$. Observe that $V(q_j)$ has been defined.
Now put
$$V(p)=\bigoplus _{j=0}^k V(q_j) \otimes_L L_j,$$
where $L_j:=L(z_1,\dots ,z_{j-1},z_{j+1},\dots ,z_k)[z_j]$. Now the
action of $\ol{\alpha} _{p,q_{\ell}}$ on $V(q_j) \otimes _L L_j $ is
given by multiplication by $z_{\ell}$, that is
$$(v_j\otimes \lambda _j )\tau_{i+1}(\ol{\alpha} _{p,q_{\ell}})= v_j\otimes \lambda_jz_{\ell}$$
for $v_j\in V(q_j)$ and $\lambda_j\in L_j$. The action of
$\alpha_{p,q_{\ell}}$ on $V(q_j)\otimes _L L_j$, with $j\ne \ell$,
is given by
$$( v_j\otimes \lambda _j)\tau_{i+1}(\alpha_{p,q_{\ell}})= v_j\otimes \lambda_jz_{\ell}^{-1}.$$
The action of $\alpha_{p,q_{\ell}}$ on $ V(q_{\ell})\otimes
L_{\ell}$ is given as follows:
$$(v_{\ell}\otimes (f_0+f_1z_{\ell}+\cdots +f_v(z_{\ell})^v))\tau_{i+1}(\alpha_{p,q_{\ell}})=
v_{\ell}\otimes (f_1+\cdots +f_v(z_{\ell})^{v-1}),$$ where $f_b\in
L(z_1,\dots ,z_{\ell -1},z_{\ell +1},\dots , z_k)$. Write
$K_j=L(z_1,\dots,z_{j -1},z_{j +1},\dots ,z_k)$. Note that
$e(p,q_j)=e(p)-\alpha _{p,q_j}\ol{\alpha }_{p,q_j}$ projects $V(p)$
onto $V(q_j)\otimes _L K_j$. Thus define $\tau _{i+1}(\beta
_{p,q_j})$ as $0$ on  $(\oplus _{q\in \cup _{j=0}^{i+1} \mathbb
P^j}V(q))(1-e(p,q_j))$, and as the isomorphism from $V(q_j)\otimes_L
K_j$ onto $V(q_j)$ given by the composition
\begin{equation}
\begin{CD}
V(q_j)\otimes _L L(z_1,\dots,z_{j -1},z_{j +1},\dots ,z_k) @>
1\otimes \sigma ^j>> V(q_j)\otimes _L L @>\cong>> V(q_j),
\end{CD}
\end{equation}
where $\sigma ^j$ is the isomorphism given by $(z_{\ell})\sigma
^j=t_{\sigma _j(\ell)}$ and $(t_u)\sigma^j=t_{u+k-1}$. The action of
$\ol{\beta} _{p,q_j}$ is now determined by the rules
(\ref{eq:A.14}).

It is straightforward to show that $\tau _{i+1}$ preserves the
defining relations for the algebra $\mathcal A \Sigma _1^{-1}$ (see
Definition \ref{defi: Q(M)}). We have to check that $\tau _{i+1}(f)$
is an invertible endomorphism of $V(p)$ for each $f\in \Sigma (p)$.
Let us fix $j$ with $1\le j\le k$, and write
$$f=f_0(\alpha _{p,q})+\alpha _{p,q_j}f_1(\alpha _{p,q})+\cdots +(\alpha _{p,q_j})^uf_u(\alpha _{p,q}),$$
where $f_b\in L[z_1,\dots ,z_{j-1},z_{j+1},\dots ,z_k]$. Let
$g\mapsto \ol{g}$ denote the involution on $K_j$ which is the
identity on $L$, and sends $z_{\ell}$ to $z_{\ell}^{-1}$ for
$\ell\ne j$. Let $\phi_j\in \text{End}_{K_j}(L_j) $ denote the map
defined by
$$(f_0+f_1z_{j}+\cdots +f_vz_{j}^v)\phi_j=
f_1+\cdots +f_vz_{j}^{v-1},$$ for $f_b\in K_j$.
 Then
$\tau_{i+1}(f)$ acts on $V(q_j)\otimes_L L_j$ by
$$(v_j\otimes \lambda _j)\tau_{i+1}(f)= v_j\otimes [\lambda _j\ol{f}_0+(\lambda _j\ol{f}_1)\phi _j+
\cdots + (\lambda _j\ol{f}_u)\phi _j^u].$$ Since
$\ol{f}_0+\ol{f}_1z_j+\cdots + \ol{f}_uz_j^u\in K_j[z_j]$ has
constant term $\ol{f}_0\ne 0$ (because $f\in \Sigma (p)$), it is
invertible in $K_j[[z_j]]$. Let
$$G=\sum _{a=0}^{\infty} g_az_j^a\in K_j[[z_j]]$$
be the inverse of $\ol{f}_0+\ol{f}_1z_j+\cdots + \ol{f}_uz_j^u$.
Then it is readily seen that $\sum _{a=0}^{\infty} g_a\phi _j ^a$
defines an endomorphism on $V(q_j)\otimes_L L_j$ which is the
inverse of the restriction of $\tau _{i+1}(f)$ to $V(q_j)\otimes_L
L_j$. This shows that $\tau _{i+1}(f)$ is an invertible endomorphism
of $V(p)$, as thus we obtain a representation $\tau _{i+1}\colon
Q_K(\cup _{j=0}^{i+1} \mathbb P ^j)\to \text{End}_K(\oplus _{p\in
\cup _{j=0}^{i+1} \mathbb P^j}V(p))$. This completes the induction
step.

In this way, we obtain a representation $\tau =\tau _{\mathbb
P}\colon Q_K(\mathbb P)\to \text{End}_K(V(\mathbb P))$. Observe
that, if $A$ is a lower subset of $\mathbb P$, then $V(A)=\oplus
_{p\in A}V_{\mathbb P}(p)$ and $\tau _A(x)=\tau _{\mathbb P}(\psi
_A(x))_{|V(A)}$, where $\psi _A\colon Q_K(A)\to e(A)Q_K(\mathbb
P)e(A)$ is the canonical map.

Now we are going to show that the representation $\tau$ is faithful.
Assume that $x$ is a nonzero element of $Q_K(\mathbb P)$. By Lemma
\ref{lem:injectivity}, there are $z_1,z_2$ in $Q_K(\mathbb P)$ such
that $z_1xz_2$ has a nonzero coefficient in $Q_K(\mathbb P)_{(p,p)}$
for some $p\in \mathbb P$, and all other nonzero coefficients of
$z_1xz_2$ belong to $\sum Q_K(\mathbb P)_{(\gamma_1,\gamma _2)}$,
where in the above sum $(\gamma_1,\gamma _2)$ ranges on all pairs of
non-trivial paths starting in $p$ and ending in a common vertex.
There are $f_1,f_2\in \Sigma (p)$ such that $f_1(z_1xz_2)f_2$ can be
written as $y_1+y_2$ with $y_1=\sum _{i=1}^s \lambda _i\got m _i\in
Q_K(\mathbb P)_{(p,p)}\setminus \{0\}$ and $y_2=\sum _{u=1}^{n_p}
\sum _{w=0}^{M_u}\alpha _{p,q_u}^w\beta_{p,q_u}y_{uw}$
 for some $y_{uw}\in Q_K(\mathbb P)$.

 It is enough to show that $\tau(f_1z_1xz_2f_2)\ne 0$.
 Take a natural number $N$ such that
 $N$ is bigger than $M_1$ and such that $N-n_i(1)\ge 0$ for all $i$, where $n_i(1)$
 is the exponent of $z_1$ in the monomial $\got m _i(z)$.
For a nonzero $v\in V(q_1)$, we consider the element $v\otimes
z_1^N$ in $V(q_1)\otimes_L L_1$, and we compute
$$(v\otimes z_1^N)\tau(f_1z_1xz_2f_2)=(v\otimes z_1^N)\tau(y_1)
+(v\otimes z_1^N)\tau (y_2)= v\otimes (\sum _{i=1}^s \lambda _i
z_1^N\ol{\got m}_i(z))\ne 0.$$ This shows that $\tau$ is faithful.

If $A$ is a lower subset of $\mathbb P$, then the map
$\psi_A\colon Q_K(A)\to e(A)Q_K(\mathbb P)e(A)$ is surjective by
Lemma \ref{normalform}. Since  $\tau _A(x)=\tau _{\mathbb P}(\psi
_A(x))_{|V(A)}$ for $x\in Q_K(A)$, and $\tau _A$ is faithful, we
conclude that $\psi _A$ is injective, and thus $\psi _A$ is an
isomorphism.
\end{proof}

The distributive lattice $\mathcal L (\mathbb P)$ of lower subsets
of $\mathbb P$, which agrees with the lattice $\mathcal L (M)$ of
order-ideals of $M=M(\mathbb P)$, can be seen now as a {\it
sublattice} of $\mathcal L (Q_K(\mathbb P))$.

\begin{prop}
\label{prop:ideallattice} Let $M$ be a finitely generated primitive
monoid such that all primes of $M$ are free. Let $A\in\mathcal L
(\mathbb P (M))$ be a lower subset of $\mathbb P (M)$ and consider
the idempotent $e(A)=\sum _{q\in A}e(q)\in Q_K(\mathbb P)$. Then the
following properties hold:
\begin{enumerate}
\item Let $I(A)$ be the ideal of $Q_K(\mathbb P)$ generated by $e(A)$.
Then $$I(A)=\bigoplus _{(\gamma _1,\gamma _2)\in \mathcal T
^2:r(\gamma _1)=r(\gamma _2)\in A}Q_K(\mathbb P)_{(\gamma _1,\gamma
_2)}.$$
\item The assignment $$A\mapsto I(A)$$
defines an injective lattice homomorphism from $\mathcal L (\mathbb
P)$ into $\mathcal L (Q_K(\mathbb P))$.
\end{enumerate}
\end{prop}

\begin{proof}
(1) This is clear from Lemma \ref{normalform}.

(2) It is clear from (1) that the map $\mathcal L (\mathbb P)\to
\mathcal L (Q_K(\mathbb P))$ is injective. Observe that
$$I(A)=\oplus_{a\in A} Q_a,$$
where $Q_a=\bigoplus _{(\gamma _1,\gamma _2)\in \mathcal T
^2:r(\gamma _1)=r(\gamma _2)=a} Q_K(\mathbb P)_{(\gamma _1,\gamma
_2)}$. It follows that $I(A)+I(B)=I(A\cup B)$ and $I(A)\cap
I(B)=I(A\cap B)$. Thus the above map is a lattice homomorphism.
\end{proof}

\begin{remark}
\label{cautionaryrem} It follows from Theorem \ref{main} and
Propositions \ref{prop:wellknownideals} and \ref{wellknownlattice}
that the map $\mathcal L (\mathbb P)\to \mathcal L (Q_K(\mathbb P))$
of Corollary \ref{prop:ideallattice}(2) is indeed a lattice
isomorphism.
\end{remark}

\begin{remark}
\label{remark:quotients} If $A$ is a lower subset of $\mathbb P$,
then $Q_K(M(\mathbb P))/I(A)$ is not in general isomorphic with
$Q_K(M(\mathbb P\setminus A))$. However the structure of
$Q_K(M(\mathbb P))/I(A)$ is clear from the presentation given in
Definition \ref{defi: Q(M)}. Indeed the generators and relations for
$Q_K(M(\mathbb P))/I(A)$ are the same as in \ref{defi: Q(M)} for all
the primes $p\in \mathbb P\setminus A$ such that $\rL(p)\cap
A=\emptyset$, and if $p\in \mathbb P\setminus A$ and $\rL(p)\cap
A\ne \emptyset$, then one has to add additional relations
$e(p,q)=\beta _{p,q}=0$ for $q\in \rL(p)\cap A$. Of course we can
omit in the presentation all idempotents $e(p)$ with $p\in A$,
because these idempotents are $0$ in $Q_K(M(\mathbb P))/I(A)$.
\end{remark}

\section{Pullbacks}
\label{sect:pullbacks}

In this section we develop part of the machinery used in the proof
of our main result, concerning pullbacks of von Neumann regular
rings. It is important to have in mind that, although  a pullback of
regular rings is always regular, the corresponding diagram at the
level of monoids of f.g. projectives might not be a pullback in the
category of monoids. We give necessary and sufficient conditions in
order for this property to hold in $K$-theoretic terms.

\begin{prop}
\label{prop:regs} Let $Q_1$ and $Q_2$ be two von Neumann regular
(resp. exchange) rings and let $\pi _i\colon Q_i\to R$ be surjective
homomorphisms. Consider the pullback
$$
\begin{CD}
P @>\rho _1>> Q_1\\
   @V\rho_2 VV  @VV\pi _1V\\
Q_2 @>\pi_2>> R
\end{CD}
$$
in the category of rings. Then  $R$ and $P$ are von Neumann
regular (resp. exchange) rings. If $Q_1$ and $Q_2$ are (strongly)
separative then $R$ and $P$ are (strongly) separative too.
\end{prop}

\begin{proof} Since $Q_1$ is a von Neumann regular (resp.
exchange) ring and $\pi _1$ is surjective, $R$ is von Neumann
regular (resp. exchange) (\cite[Lemma 1.3]{vnrr},
\cite[Proposition 1.4]{Nicholson}). Let $I_i\lhd Q_i$ be the
kernel of $\pi _i$, $i=1,2$. We have a commutative diagram with
exact rows
$$
\begin{CD}
0 @>>> I_2 @>>> P @>\rho _1>> Q_1 @>>> 0 \\
& & @V=VV  @V\rho_2 VV  @VV\pi _1V  \\
0 @>>> I_2  @>>> Q_2 @>\pi_2>> R @>>> 0
\end{CD}
$$
Assume that $Q_1$ and $Q_2$ are von Neumann regular. Then $I_2$
and $Q_1$ are von Neumann regular, and it follows from \cite[Lemma
1.3]{vnrr} that $P$ is also von Neumann regular. If $Q_1$ and
$Q_2$ are exchange rings then $I_2$ and $Q_1$ are also exchange
rings (see \cite[Theorem 2.2]{Aext}), and idempotents of $Q_1$ can
be lifted to idempotents of $P$, because, being $Q_2$ an exchange
ring, idempotents in $R$ can be lifted to idempotents in $Q_2$. It
follows from \cite[Theorem 2.2]{Aext} that $P$ is an exchange
ring.

If $Q_1$ and $Q_2$ are (strongly) separative then $R$ is
(strongly) separative because it is a factor ring of a (strongly)
separative exchange ring. Also \cite[Theorem 4.5]{AGOP}
 (\cite[Theorem 5.2]{AGOP}) shows that $P$ is (strongly) separative.
\end{proof}

Let $Q_1,\dots ,Q_k$ be rings, and let $\pi_i\colon Q_i\to R$ be
surjective homomorphisms. Let $\rho _i\colon P\to Q_i$ be the
limit (pullback) of the morphisms $\pi_i\colon Q_i\to R$. We have
$k$ index maps $\partial _i\colon K_1(R)\to K_0(I_i)$ for
$i=1,\dots ,k$, where $I_i$ is the kernel of $\pi _i$. If
$e=(e_1,\dots ,e_k)$ is an idempotent in $P$, then we get
corresponding maps $\partial _i\colon K_1(\pi _i(e_i)R\pi
_i(e_i))\to K_0(e_iI_ie_i)$. These maps fit into an exact sequence
$$
\begin{CD}
K_1(e_iQ_ie_i) @>(\pi _i)_*>> K_1(\pi _i(e_i)R\pi_i(e_i))
@>\partial _i>> K_0(e_iI_ie_i) @>>> K_0(e_iQ_ie_i).
\end{CD}
$$
We are now ready to state the main result of this section.
 Note that
the $K$-theoretic part of the result below might be considered as a
nonstable version of \cite[Theorem 3.3]{Milnor}.

\begin{theor}
\label{pullbacktheorem} Let $Q_1,\dots ,Q_k$ be (strongly)
separative von Neumann regular (resp. exchange) rings, and let
$\pi_i\colon Q_i\to R$ be surjective homomorphisms. Let $\rho
_i\colon P\to Q_i$ be the limit (pullback) of the morphisms
$\pi_i\colon Q_i\to R$. Then $P$ is a (strongly) separative von
Neumann regular (resp. exchange) ring, and the maps $\mon{\rho
_i}\colon \mon{P}\to \mon{Q_i}$ are the limit of the family of maps
$\mon{\pi _i}\colon \mon{Q_i}\to \mon{R}$ in the category of monoids
if and only if for each idempotent $e=(e_1,\dots ,e_k)$ in $P$, we
have that for every $i=1,\dots ,k$, \begin{equation}
\label{eq:pullback} K_1(\pi _i(e_i)R\pi
_i(e_i))=(\pi_i)_*(K_1(e_iQ_ie_i))+\big(\bigcap _{j\ne
i}(\pi_j)_*(K_1(e_jQ_je_j))\big).
\end{equation} It is enough to check the
above condition for some generator $e=(e_1,\dots ,e_k)$ of each
finitely generated trace ideal of $P$.
\end{theor}

\begin{proof}
The first part follows easily by induction from Proposition
\ref{prop:regs}.

Assume now that $Q_1, \dots ,Q_k$ are just exchange rings. (The
separativity will be used later in the proof.) Let $\tau _i\colon
M\to \mon{Q_i}$ be the limit of the family of maps $\mon{\pi
_i}\colon \mon{Q_i}\to \mon{R}$ in the category of monoids. Recall
that
$$M=\{(x_1,\dots ,x_n)\in \prod \mon{Q_i}: \mon{\pi_i}(x_i)=\mon{\pi _j}(x_j),
 \text{ for all } i,j\}.$$
Obviously we have a unique monoid homomorphism $\rho\colon
\mon{P}\to M$ such that $\tau _i\circ \rho =\mon{\rho _i}$ for all
$i$. We show that the map $\rho$ is always surjective. (We don't
need the extra condition in the statement to prove this.
Indeed, by the results in \cite[$\S 2$]{Milnor}, the surjectivity
holds even without the hypothesis that the rings $Q_i$ are exchange
rings.)
Let $P'$ be the limit of the family of maps $\pi_i\colon Q_i\to R$,
for $1\le i\le k-1$, and let $M'$ be the limit of the family
$\mon{\pi _i}\colon \mon{Q_i}\to \mon{R}$, $1\le i\le k-1$. Let
$\rho '\colon \mon{P'}\to M'$ be the canonical map. Assume that
$([e_1],[e_2],\dots ,[e_k])$ is an element in $M$, with $e_i\in
\text{Idem}(M_n(Q_i))$ for some $n\ge 1$. By induction, there is
$f\in M_m(P')$, for some $m\ge n$ such that
$$\rho '([f])=([e_1],\dots ,[e_{k-1}])\in M'.$$

Replacing each $e_i$ with $e_i\oplus 0_{m-n}$, we can assume that
all $e_i\in M_m(Q_i)$ for all $i$. Observe that
$$\pi _k(e_k)\sim \pi _1(f_1)=\cdots =\pi_{k-1}(f_{k-1}),$$
where $f=(f_1,\dots ,f_{k-1})$. By \cite[Lemma 3.1(a)]{Aext},
there exists an idempotent $q$ in $M_m(Q_k)$ such that $q\sim
e_k-e_k'$ for some idempotent $e_k'\in e_kM_m(I_k)e_k$, and $\pi
_k(q)=\pi _1(f_1)$. Consider the idempotent
$$e=(f_1\oplus 0,\dots ,f_{k-1}\oplus 0, q\oplus e_k')\in \prod M_{2m}(Q_i).$$
Then $\pi _1(f_1\oplus 0)=\cdots = \pi_{k-1}(f_{k-1}\oplus 0)=\pi
_k(q\oplus e_k')$ so that $e\in M_{2m}(P)$ and clearly $\rho
([e])=([e_1],\dots ,[e_k])$, showing that $\rho$ is surjective.

Assume first that the maps $\mon{\rho _i}\colon \mon{P}\to
\mon{Q_i}$ are the limit of the family of maps $\mon{\pi _i}\colon
\mon{Q_i}\to \mon{R}$ in the category of monoids, and let us show
(\ref{eq:pullback}). This implication does not use separativity.

Assume that $[u]\in K_1(R)$, with $u\in GL_n(R)$ for some $n\ge
1$. Then $\partial _1 ([u])=[1-yx]-[1-xy]$, where $x,y\in
M_n(Q_1)$ and $x=xyx$, $y=yxy$, and $\pi _1 (x)=u$. (Such a
lifting of $u$ exists by \cite[Lemma 2.1]{Aext}.) Put $e_1=yx$ and
$e_1'=xy$. Then we have $\pi _1(e_1)=\pi _1(e_1')=1_n$. Consider
the idempotents $(e_1,1_n,\dots ,1_n)$ and $(e_1',1_n,\dots ,1_n)$
in $M_n(P)$, where $1_n$ is the identity matrix of size $n\times
n$. Clearly
$$\rho ([(e_1,1_n,\dots ,1_n)])=\rho ([(e_1',1_n,\dots ,1_n)])\in M\, . $$
Since $\rho $ is an isomorphism,  we get $(e_1,1_n,\dots ,
1_n)\sim (e_1',1_n,\dots ,1_n)$ in $M_n(P)$. We get elements
$$(y_1,z_2,\dots ,z_k)\in (e_1,1_n,\dots,1_n)P(e_1',1_n,\dots
,1_n)$$ and $$(x_1,t_2,\dots ,t_k)\in (e_1',1_n,\dots
,1_n)P(e_1,1_n,\dots ,1_n)$$ such that $$(y_1,z_2,\dots,
z_k)(x_1,t_2,\dots ,t_k)=(e_1,1_n,\dots ,1_n)$$ and
$$(x_1,t_2,\dots ,t_k)(y_1,z_2,\dots ,z_k)=(e_1',1_n,\dots ,1_n).$$
So we get that $t_i\in GL_n(Q_i)$ and $\pi _1(x_1)=\pi_i(t_i)$ for
all $i\ge 2$, so that $$[\pi _1(x_1)]\in \bigcap _{i\ne 1}(\pi
_i)_*(K_1(Q_i)).$$ Furthermore
$$\partial _1 ( [\pi _1(x_1)])=
[1-y_1x_1]-[1-x_1y_1]=[1-e_1]-[1-e_1']=\partial _1 ([u]).$$ By
exactness of the $K$-theory sequence, we get that $[u]-[\pi
_1(x_1)]\in (\pi_1)_*(K_1(Q_1))$, so that $[u]\in (\pi
_1)_*(K_1(Q_1))+\bigcap _{i\ne 1}(\pi _i) _*(K_1(Q_i))$. This
shows that
$$K_1(R)= (\pi _1)_*(K_1(Q_1))+\bigcap _{i\ne 1}(\pi _i)
_*(K_1(Q_i)).$$

If $e=(e_1,\dots ,e_k)$ is an idempotent in $P$ then the maps ${\rho
_i}_| \colon ePe\to e_iQ_ie_i$ give the limit of the family of maps
${\pi _i}_|\colon e_iQ_ie_i\to \pi _i(e_i)R\pi _i (e_i)$, and the
maps $\mon{{\rho_i}_|}\colon \mon{ePe}\to \mon{e_iQ_ie_i}$ give the
limit of the family $\mon{{\pi _i}_|}\colon \mon{e_iQ_ie_i}\to
\mon{\pi _i(e_i)R\pi _i(e_i)}$ in the category of monoids. (Observe
that, if $e$ is an idempotent in a general ring $T$, then we can
identify $\mon{eTe}$ with the order-ideal $\mon{TeT}$ of $\mon{T}$
consisting of classes $[g]\in \mon{R}$ such that $g\in
M_{\infty}(TeT)$; see \cite[proof of Lemma 7.3]{AF}.) By the above
argument, we get (\ref{eq:pullback}).

Now assume that (\ref{eq:pullback}) holds for every idempotent
$e=(e_1,\dots ,e_k)$ in $P$ and every index $i$. Assume that
$Q_1,\dots ,Q_k$ are separative exchange rings. Since $\rho\colon
\mon{P}\to M$ is always surjective, we only have to show that it is
injective.  Let $(e_1,\dots ,e_k)$ and $(e_1',\dots ,e_k')$ be
idempotents in $M_n(P)$ for some $n\ge 1$ such that $e_i\sim e_i'$
in $M_n(Q_i)$ for all $i$. By using standard arguments, and passing
to matrices of bigger size, we can assume that
$e_k=u_ke_k'u_k^{-1}$, where $u_k\in E_m(Q_k)$, the group of
elementary matrices of size $m\times m$ for some $m\ge n$, see
\cite[ 1.2.1, 2.1.3]{Rosenberg}. Now since all elementary matrices
lift, there is $u_i\in E_m(Q_i)$ such that $\pi _i(u_i)=\pi
_k(u_k)$, so that $u:=(u_1,\dots ,u_{k-1},u_k)\in GL_m(P)$ and
$u(e_1',\dots ,e_k')u^{-1}=(e_1'',\dots ,e_{k-1}'',e_k)$, so we can
assume that $e_k=e_k'$.

Thus assume we have two idempotents $(e_1,\dots ,e_{k-1},e_k)$ and
$(e_1',\dots ,e'_{k-1}, e_k)$ in $M_n(P)$ such that $e_i\sim e_i'$
in $M_n(Q_i)$ for $i=1,\dots ,k-1$. We have $\pi_i (e_i)=\pi
_k(e_k)=\pi _j(e_j')$ for all $i,j$. Replacing each ring $T$ in
the diagram with $M_n(T)$, we can assume that $n=1$. (The validity
of (\ref{eq:pullback}) for idempotents in $M_n(P)$ is indeed
justified by the last sentence in the statement: the trace ideal
of $P$ generated by an idempotent $E$ in $M_n(P)$ is generated by
a single idempotent of $P$,  see the last paragraph in the proof.)
By using induction, we can write $(e_1,\dots ,e_{k-1})=xy$ and
$(e_1',\dots ,e_{k-1}')=yx$ with $x\in (e_1,\dots
,e_{k-1})P'(e_1',\dots ,e_{k-1}')$ and $y\in (e_1',\dots
,e_{k-1}')P'(e_1,\dots ,e_{k-1})$, where $P'$ is the limit of the
family of maps $\pi _i\colon Q_i\to R$, $1\le i\le k-1$. We have
$$\pi _1(x_1)\pi_1 (y_1)=\pi _1(e_1)=\pi _1(e_1')=\pi _1(y_1)\pi _1(x_1)\, ,$$
so that $\pi _1(x_1)=\pi _i(x_i)\in GL_1(\pi_1(e_1)R\pi _1(e_1))$,
for $1\le i\le k-1$,  and $\pi _1(y_1)=\pi_1(x_1)^{-1}$. Take
$z_i,t_i\in e_iQ_ie_i$, $1\le i\le k-1$,  such that
$z_i=z_it_iz_i$ and $t_i=t_iz_it_i$, with $\pi _i(z_i)=\pi
_i(x_i)$, so that $\pi _i(t_i)=\pi _i(y_i)$. Set $z=(z_1,\dots
,z_{k-1})$ and $t=(t_1,\dots ,t_{k-1})$, and note that $z,t\in
fP'f$, where $f=(e_1,\dots ,e_{k-1})$. Put $f'=(e_1',\dots
,e_{k-1}')\in P'$.
 Write $h:=zt$ and $h':=tz$. Then $h,h'\in fP'f$ and
$$(f, e_k)=(h,e_k)+(f-h,0)$$
with $f-h\in \prod_{i=1}^{k-1} I_i$. On the other hand,
$$(f',e_k)=(yhx, e_k)+(f'-yhx, 0),$$
and clearly $(f-h,0)\sim (f'-yhx,0)$ in $P$, so it suffices to
check that $(h,e_k)\sim (yhx,e_k)$ in $P$. Let $\pi'\colon P'\to
R$ be the canonical map. Since $\pi'(yz)=\pi'(tx)=\pi _k(e_k)$, we
get
$$(yhx, e_k)=(yz,e_k)
(h',e_k)(tx,e_k)$$ so that $(yhx, e_k)\sim (h',e_k)$ in $P$.

Therefore we only have to show that $(h,e_k)\sim (h',e_k)$ in $P$.
By applying our hypothesis (with the idempotent $(h,e_k)$), we can
write $[\pi '(z)]=\alpha +\beta$ in $K_1(\pi'(h)R\pi '(h))$, where
$\alpha\in (\pi _k)_*(K_1(e_kQ_ke_k))$ and $\beta\in \bigcap
_{i=1}^{k-1} (\pi _i)_*(K_1(h_iQ_ih_i))$.

By \cite{AGOR}, we can choose $v$ in $GL_1(e_kQ_ke_k)$ such that
$(\pi _k)_*([v])=\alpha$, so that $[\pi '(z)\pi_k(v)^{-1}]=\beta$.
Since $\partial _i(\beta )=0$ for $1\le i\le k-1$, it follows from
\cite[Theorem 2.4]{Perera} that there exists a unit $u\in
GL_1(hP'h)$ such that $\pi'(u)=\pi '(z)\pi _k(v)^{-1}$, so that
$(u^{-1}z,v)\in P$ and
$$(u^{-1}z,v)(h',e_k)(tu,v^{-1})=(u^{-1}zh'tu,ve_kv^{-1})=(u^{-1}hu,e_k)=(h,e_k)$$
and $$(tu,v^{-1})(h,e_k)(u^{-1}z,v)=(h',e_k)$$ and we get that
$(h',e_k)\sim (h,e_k)$ in $P$, as desired.

The last sentence in the statement of the theorem follows from
Morita invariance of $K_1$. Observe that, being $P$ an exchange
ring, the finitely generated trace ideals of $P$ are the ideals
generated by finitely many idempotents (cf. \cite[page 377]{AF})
and so, by \cite[Lemma 2.1]{AGOR}, they are the ideals generated
by a single idempotent.
\end{proof}

\begin{exem}
\label{exam:pullbackchecked} To illustrate Theorem
\ref{pullbacktheorem} in the context of our construction, we check
that the conditions in the theorem are satisfied for the pullback
diagram (\ref{diagram-example}) in Example \ref{exam:main}. We may
assume that $(e_1,e_2)=(1_{Q_1},1_{Q_2})$. We have that
$(\pi_1)_*(K_1(Q_1))=(K(t_2,t_3,\dots, ))[t_1]_{(t_1)}$ and
$(\pi_2)_*(K_1(Q_2))=(K(t_1,t_3,\dots ,))[t_2]_{(t_2)}$, so that
$$K_1(L)=L^{\times}=(\pi_1)_*(K_1(Q_1))+(\pi _2)_*(K_1(Q_2)),$$
as desired. Since $Q_1$ and $Q_2$ are strongly separative von
Neumann regular rings (see \cite[4.3]{AB2}), we conclude from
Theorem \ref{pullbacktheorem} that $Q_K(\mathbb P)$ is a strongly
separative von Neumann regular ring and that $\mon{Q_K(\mathbb
P)}\cong \langle p,a,b\mid p=p+a=p+b\rangle$ (cf. Proposition
\ref{pullfgr}).
\end{exem}

\section{Pushouts}
\label{sect:pushout}

In this section we discuss some constructions of monoids and rings
associated to a certain class of pushouts of monoids. The section
also includes some needed computation on pullbacks (Proposition
\ref{pullfgr}). We refer the reader to \cite[Chapter 8]{Howie} for
the general theory of free products with amalgams of semigroups.

\begin{lem}\label{pushoutlemma}
Let $M$ and $N$ be two conical monoids. Let $I$ be an order-ideal
of $M$ which is isomorphic with an order-ideal of $N$, through an
isomorphism $\varphi\colon  I\to \varphi (I)\lhd N$. Let $P$ be
the monoid $M\times N /\sim$ where $\sim $ is the congruence on
$M\times N$ generated by all pairs $((s,0), (0,\varphi (s)))$ for
$s\in I$. Then we have a pushout diagram
$$
\begin{CD}
I @>>> M\\
   @V\varphi VV  @VV\iota _1V\\
N @>\iota _2>> P
\end{CD}
$$
where $\iota _1\colon M\to P$ is the map $\iota _1(x)=[(x,0)]$ and
$\iota _2\colon N\to P$ is the map $\iota _2(y)=[(0,y)]$. The maps
$\iota _1$ and $\iota _2$ are injective and $I':=\iota _1(I)=\iota
_2(\varphi (I))$ is an order-ideal of $P$ such that $P/I'\cong
M/I\times N/\varphi (I)$. Moreover, if there exist injective
monoid homomorphisms $\theta _1\colon  M\to Q$ and $\theta
_2\colon N\to Q$ into a conical refinement monoid $Q$ such that
${\theta _1}_{|I}=\theta _2\circ \varphi$ and $\theta _1(M)\cap
\theta _2(N)=\theta _1(I)=\theta _2(\varphi (I))$, and $\theta
_1(M)$ and $\theta _2(N)$ are order-ideals of $Q$, then there is
an embedding $\iota \colon P\to Q$ such that $\theta _i=\iota
\circ \iota _i$ for $i=1,2$. \end{lem}
\begin{proof}
It is clear that the diagram is a pushout diagram. We will show
that the map $\iota _1$ is injective. The proof for $\iota _2$ is
similar. If $(s_1,0)\sim (s_2,0)$, then there are sequences
$x_i,y_i$, with $y_i\in I$, such that $s_1=x_1+y_1$, $y_1=x_2+y_2$
and for all $i$, $x_{2i-1}+x_{2i}=x_{2i+1}+y_{2i+1}$, and
$y_{2i}+y_{2i+1}=x_{2i+2}+y_{2i+2}$, such that
$$(s_1,0)=(x_1+y_1,0)\sim (x_1,\varphi (y_1))=(x_1,\varphi (x_2+y_2))
\sim (x_1+x_2,\varphi (y_2)) $$
$$=(x_3+y_3,\varphi (y_2))\sim
\cdots \sim (s_2,0).$$ Since $M$ is conical there is $n$ such that
$y_{2n}=0$ and $s_2=x_{2n-1}+x_{2n}$. We thus get
$$s_2=x_{2n-1}+x_{2n}= x_{2n-1}+y_{2n-2}+y_{2n-1}=\cdots =x_1+y_1=s_1 .$$
In general note that $(s,0)\sim (x,y)$, where $s\in M$ and $x\in
M$ and $y\in N$, implies $y=\varphi (x')$ for some $x'\in I$ such
that $x+x'=s$, so we get that $M$ and $I$ are order-ideals of $P$.

Since $P$ is the pushout of the diagram, we clearly have a map
$\gamma' \colon P\to M/I\times N/\varphi (I)$, defined by $\gamma
([(x,y)])=([x],[y])$. Clearly this map factors through a map
$\gamma\colon P/I'\to M/I\times N/\varphi (I)$. The inverse of
this map is provided by the well-defined map $\tau \colon
M/I\times N/\varphi (I)\to P/I'$, given by $\tau
([x],[y])=[(x,y)]\in P/I'$.

Finally, assume that there exist injective monoid homomorphisms
$\theta _1\colon  M\to Q$ and $\theta _2\colon N\to Q$ into a
conical refinement monoid $Q$ such that ${\theta _1}_{|I}=\theta
_2\circ \varphi$ and $\theta _1(M)\cap \theta _2(N)=\theta
_1(I)=\theta _2(\varphi (I))$, and $\theta _1(M)$ and $\theta _2(N)$
are order-ideals of $Q$. By the pushout property there is a monoid
homomorphism $\iota \colon P\to Q$ such that $\theta _i=\iota \circ
\iota _i$ for $i=1,2$. It remains to show that $\iota $ is
injective. So assume that for $[(x_1,y_1)], [(x_2,y_2)]\in P$ we
have $\theta _1(x_1)+\theta _2(y_1)=\theta_1(x_2)+\theta _2(y_2)$.
Applying refinement in $Q$ we get $\theta _1(x_1)=z_1+z_2$ and
$\theta _2(y_1)=t_1+t_2$ such that $z_1+t_1=\theta _1(x_2)$ and
$z_2+t_2=\theta _2(y_2)$. Since $\theta _1(M)$ is an order-ideal in
$Q$, we get $z_1,z_2\in \theta _1(M)$, and similarly $t_1,t_2\in
\theta _2(N)$. So we can write $z_i=\theta _1(v_i)$ and $t_i=\theta
_2(w_i) $, where $v_i\in M$ and $w_i\in N$. Since $\theta _i$ are
injective we get $x_1=v_1+v_2$ and $y_1=w_1+w_2$. Now $\theta
_1(v_2)\in \theta _1(M)\cap \theta _2(N)=\theta _1(I)$, so we get
$v_2\in I$. Similarly we get $w_1\in \varphi (I)$ and also
$x_2=v_1+\varphi ^{-1}(w_1)$ and $y_2=\varphi (v_2)+w_2$. We have
$$[(x_1,y_1)]=[(v_1+v_2,w_1+w_2)]=[(v_1+\varphi ^{-1}(w_1), \varphi (v_2)+w_2)]
=[(x_2,y_2)].$$
This shows that $\iota $ is injective, as desired.
\end{proof}

We now study the notion of a crowned pushout, that we will need in
our main construction.

Let $P$ be a conical monoid.  Suppose that $P$ contains
order-ideals $I$ and $I'$, with $I\cap I'=0$, such that there is
an isomorphism $\varphi \colon I\to I'$. We have a diagram
$$\begin{CD}
I @>=>> I\\
   @V\varphi VV  @VV\iota_1 V\\
I' @>\iota_2>> P
\end{CD}
$$
which is not commutative.

The {\it crowned pushout} $Q$ of $(P,I,I',\varphi)$ is by
definition the coequalizer of the maps $\iota_1\colon I\to P$ and
$\iota _2\circ \varphi\colon I\to P$, so that there is a map
$f\colon P\to Q$ with $f(\iota _1(x))=f(\iota _2(\varphi (x)))$
for all $x\in I$ and given any other map $g\colon P\to Q'$ such
that $g(\iota _1(x))=g(\iota _2(\varphi (x)))$ for all $x\in I$,
we have that $g$ factors uniquely through $f$.

In the situation of Lemma \ref{pushoutlemma}, the pushout can be
obtained as the crowned pushout of  $(M\times N,\, I\times 0,\,
0\times \varphi (I),\, \varphi)$.

\begin{prop}
\label{crownpushmon} Let $P$ be a conical refinement monoid.
Suppose that $P$ contains order-ideals $I$ and $I'$, with $I\cap
I'=0$, such that there is an isomorphism $\varphi \colon I\to I'$.
Let $Q$ be the crowned pushout of $(P,I,I', \varphi)$. Then $Q$ is
the monoid $P/\sim$ where $\sim$ is the congruence on $P$
generated by $x+i\sim x+\varphi (i)$ for $i\in I$ and $x\in P$.
Moreover $Q$ is a conical refinement monoid, and $Q$ contains an
order-ideal $Z$, isomorphic with $I$, such that the projection map
$\pi \colon P\to Q$ induces an isomorphism $P/(I+I')\cong Q/Z$.

\end{prop}

\begin{proof}
It is clear that the canonical projection $\pi \colon P\to P/\sim$
is the coequalizer of the maps $\iota _1$ and $\iota _2\circ
\varphi$.

It is straightforward to show that $\sim$ is a refining relation
on $P$, that is, if $\alpha \sim x_1+x_2$, then we can write
$\alpha =\alpha _1+\alpha _2$ with $\alpha _i\sim x_i$. Since $P$
is a refinement monoid, it follows that $Q$ is a refinement monoid
too. Clearly $Q$ is conical.

Note that $I+I'\cong I\times I'$, because $P$ is a conical
refinement monoid. By Lemma \ref{pushoutlemma}, the pushout $Z$ of
$I\leftarrow I\rightarrow I'$ is obtained as $I\times I' /\sim '$,
where $\sim '$ is the restriction of the congruence $\sim$ to
$I\times I'\cong I+I'$. It follows easily that $Z$ is an
order-ideal of $Q:=P/\sim$. By Lemma \ref{pushoutlemma}, the map
$I\to Z$ is an isomorphism.

Since the projection map $P\to P/(I+I')$ clearly equalizes the
maps $\iota _1\colon I\to  P$ and $\iota _2\circ \varphi \colon
I\to P$, we get a unique factorization $P\to Q\to P/(I+I')$. We
get therefore a factorization of the identity map
$$P/(I+I')\to Q/Z\to P/(I+I').$$
Since the map $P/(I+I')\to Q/Z$ is surjective we get that the map
$$P/(I+I')\longrightarrow Q/Z$$
is an isomorphism.
\end{proof}

\begin{lem}\label{lem:quotient}{\rm (\cite[Lemma 6.3]{APW})}
Let $I$ be an order-ideal of a primitive monoid $M$. Then $M/I$ is
also a primitive monoid and the canonical map $\pi \colon M\to
M/I$ induces a $\lhd$-isomorphism from $\mathbb P(M)\setminus
\mathbb P(I)$ onto $\mathbb P(M/I)$.
\end{lem}

The following result determines the structure of the crowned
pushout for a primitive monoid.

\begin{lem}\label{structVcrownpushout}
Let $P$ be a primitive monoid.  Suppose that $P$ contains
order-ideals $I$ and $I'$, with $I\cap I'=0$, such that there is
an isomorphism $\varphi \colon I\to I'$. Let $Q$ be the crowned
pushout of $(P,I,I', \varphi)$. Then $Q$ is a primitive monoid
with $\mathbb P (Q)=\mathbb P (P)\setminus \mathbb P (I')$, with
the order relation induced by the order relation in $\mathbb P
(P)$ and the additional relations $p< q$ whenever $p\in \mathbb
P(I)$, $\, q\in \mathbb P(P)\setminus (\mathbb P (I)\sqcup \mathbb
P(I')) $ and $\varphi (p)< q$ in $\mathbb P (P)$. Moreover, for
$p\in \mathbb P (Q)$, we have that $p$ is free in $Q$ if and only
if $p$ is free in $P$.
\end{lem}

\begin{proof}
By Proposition \ref{crownpushmon}, $Q$ contains an order-ideal $Z$,
isomorphic with $I$, such that $Q/Z\cong P/(I+I')$. Moreover $Q$ is
a conical refinement monoid. To show that $Q$ is antisymmetric, we
define for $p\in \mathbb P (P)\setminus \mathbb P (I')$, some maps
$\phi ^Q_p\colon Q\to \mathbb Z^{\infty}=\mathbb Z^+\cup
\{\infty\}$, by
$$\phi _p^Q([x])=
\begin{cases}\phi _p^P(x)\quad & \text{if $p\in \mathbb P(P)\setminus
(\mathbb P(I)\sqcup \mathbb P (I'))$,} \\
               \phi _p^P(x) +\phi _{\varphi (p)}^P(x) \quad   &\text{if $p\in
\mathbb P(I)$}
      \end{cases}$$
Here $\phi^P _p(x)=\text{sup}\{n\in \mathbb Z:np\le x\}$ are the
functions associated to the primitive monoid $P$, see
(\ref{Eq:Defphi}). It is easy to show from the definition of the
relation $\sim$ that the maps $\phi ^Q_p$ are well-defined, and
that they define an injective monoid homomorphism $\phi ^Q\colon
Q\to \prod _{p\in \mathbb P (P)\setminus \mathbb P (I')}\mathbb Z
^{\infty}$, from which it follows that $Q$ is antisymmetric. Thus
we get that $Q$ is a primitive monoid.

By Lemma \ref{lem:quotient}, we have
$$\mathbb P (Q)\setminus \mathbb P (Z)=\mathbb P (Q/Z)=\mathbb P
(P/(I+I'))= \mathbb P(P)\setminus (\mathbb P (I)\sqcup \mathbb P
(I')),$$ and so $\mathbb P(Q)=\mathbb P(Z)\sqcup (\mathbb
P(P)\setminus (P (I)\sqcup \mathbb P (I'))) =\mathbb P(P)\setminus
\mathbb P(I')$, and clearly the function $\phi ^Q$ defined above
is the canonical function associated to the primitive monoid $Q$.
The rest of the proof follows from this.
\end{proof}

Now we will describe a construction with von Neumann regular rings
which produces the effect of the crowned pushout of Proposition
\ref{crownpushmon} at the level of monoids of projectives.
Moreover we have a great deal of control on the ring so produced.

We will need the theory of Morita-equivalence for rings with local
units. We refer the reader to \cite{AM} for basic information on
this theory (see also \cite{Abrams} and \cite{GS}). Recall that
$R$ is a {\it ring with local units} if every finite subset of $R$
is contained in a ring of the form $eRe$, where $e=e^2\in R$.
Every von Neumann regular ring is a ring with local units, by
\cite[Example 1]{AM}. When working with modules over rings with
local units, one assume that these modules are {\it unitary}. For
instance if $M$ is a right module over $R$, this means that
$M=MR$. Equivalently for every finite number of elements
$x_1,\dots ,x_n$ in $M$ there is an idempotent $e$ in $R$ such
that $x_ie=x_i$ for all $i$.

Let $I$ and $I'$ be two rings with local units. We say that $I$
and $I'$ are {\it Morita-equivalent} in case there are (unitary)
bimodules ${_IN}_{I'}$ and ${_{I'}M}_I$  and bimodule isomorphisms
$N\otimes_{I'}M\to I$ and $M\otimes _{I}N\to I'$ satisfying the
following associativity laws: $(nm)n'=n(mn')$ and $(mn)m'=m(nm')$,
for all $n,n'\in N$ and $m,m'\in M$.

\begin{prop}\label{crownpushoutprop}
Let $R$ be a (non-necessarily unital) von Neumann regular ring
with ideals $I$ and $I'$ such that $I\cap I'=0$, and suppose that
$I$ and $I'$ are Morita equivalent. Then there is a von Neumann
regular ring $U$ with an ideal $J$ such that the following holds:

(1) There exists an injective ring homomorphism $\alpha \colon
R\to U$ such that $\alpha (I)\subseteq J$ and $\alpha
(I')\subseteq J$. If $R$ is unital, then $U$ is unital and
$\alpha$ is a unital map.

(2) The map $\mon{\alpha}\colon \mon{R}\to \mon{U}$ restricts to an
isomorphism from $\mon{I}$ onto $\mon{J}$, and it also restricts to
an isomorphism from $\mon{I'}$ onto $\mon{J}$.

(3) Let $\varphi\colon \mon{I}\to \mon{I'}\subseteq \mon{R}$ be the
isomorphism defined by $$\varphi :=(\mon{\alpha}
_{|\mon{I'}})^{-1}\circ (\mon{\alpha } _{|\mon{I}}).$$ Then
$\mon{\alpha}\colon \mon{R} \to \mon{U}$ is the coequalizer of the
following (non-commutative) diagram:

\begin{equation}
\label{diag:NCD.1}
\begin{CD}
\mon{I} @>=>> \mon{I}\\
   @V{\varphi }VV  @VVV\\
\mon{I'} @>>> \mon{R}
\end{CD}
\end{equation}
\end{prop}

\begin{proof}
(1) Let ${_IN}_{I'}$ and ${_{I'}M}_I$ be (unitary) bimodules
implementing a Morita equivalence between $I$ and $I'$, so that
there are bimodule isomorphisms $N\otimes_{I'}M\to I$ and
$M\otimes _{I}N\to I'$ satisfying the associativity laws
$(nm)n'=n(mn')$ and $(mn)m'=m(nm')$, for all $n,n'\in N$ and
$m,m'\in M$.

The following diagram is a pullback:
$$
\begin{CD}
R @>\pi _1>> R/I\\
   @V\pi _2 VV  @VV\pi_4V\\
R/I' @>\pi _3>> R/(I+I')
\end{CD} .
$$
Consider the ring $U$ whose elements are matrices
$$X=\begin{pmatrix}
r_1+I' & n \\
m & r_2+I
\end{pmatrix}$$
such that $\pi _3(r_1+I')=\pi _4(r_2+I)$, where $r_1+I'\in R/I'$ and
$r_2+I\in R/I$, and $n\in N$ and $m\in M$. Since the diagram above
is a pullback, we see that for a matrix $X$ as above, there is a
unique $r\in R$ such that  $r_1+I'=\pi_2(r)$ and $r_2+I=\pi_1(r)$,
which gives another way of describing the elements of $U$.  Set
$J=\begin{pmatrix}
I & N \\
M & I'
\end{pmatrix}$, which is an ideal of $U$. The map $\alpha\colon R\to
U$ is defined by $\alpha (r)=\text{diag}(\pi _2(r),\pi_1(r))\in U$.
Clearly $U$ and $\alpha$ are unital if so is $R$.

It is easy to check that $J$ has local units. Since $J$ is
Morita-equivalent to $I$ (and to $I'$) and has local units, it
follows from \cite[Proposition 3.1]{AM} that $J$ is von Neumann
regular. Since $J$ and $U/J\cong R/I'\times R/I$ are von Neumann
regular, it follows from \cite[Lemma 1.3]{vnrr} that $U$ is von
Neumann regular.

(2) Observe that $\alpha (I)=\begin{pmatrix}  I & 0 \\ 0 & 0
\end{pmatrix}$, and similarly $\alpha (I')=\begin{pmatrix}
0 & 0 \\ 0 & I'
\end{pmatrix}$. Since $J$ is Morita-equivalent to both $I$ and
$I'$, we conclude that the maps $\mon{\alpha}_{|\mon{I}}$ and
$\mon{\alpha }_{|\mon{I'}}$ are both isomorphisms onto $\mon{J}$,
see \cite[Corollary 5.6]{GoodLeavitt}.

(3) The coequalizer of the maps in the diagram (\ref{diag:NCD.1}) is
constructed in Proposition \ref{crownpushmon} as $\mon{R}/\sim$
where $\sim$ is the congruence on $\mon{R}$ generated by $x+i\sim
x+\varphi (i)$ for every $x\in \mon{R}$ and $i\in \mon{I}$. Clearly
there exists a unique monoid homomorphism $\rho\colon {\mon{R}/\sim}
\longrightarrow \mon{U}$ such that $\rho ([x])=\mon{\alpha }(x)$ for
every $x\in \mon{R}$, where $[x]$ denotes the equivalence class of
$x\in \mon{R}$ in $\mon{R}/\sim$. We have to show that the map
$\rho$ is an isomorphism.

Since our construction is compatible with the passage to matrices
of arbitrary size, we can restrict our considerations to
idempotents in $R$ and $U$. The following observation will be
crucial for our proof:

Given a finite number of elements $x_1,\dots ,x_n,y_1,\dots
,y_m\in J$ and given any "diagonal" idempotent $e=\begin{pmatrix}
e_1 +I'& 0 \\
0 & e_2+I
\end{pmatrix}\in U$ such that $ex_i=x_i$ for all $i$ and $y_je=y_j$ for all $j$,
there
exist idempotents $g_1\in I$ and $g_2\in I'$ such that
$g:=\begin{pmatrix}
g_1 & 0 \\
0 & g_2
\end{pmatrix}\in J$ satisfies $g\le e$ and $gx_i=x_i$ for all
$i$ and $y_jg=y_j$ for all $j$.

This observation is easily proved taking into account that the
sets of idempotents in $I$ and $I'$ are directed, and the facts
that $N=IN=NI'$ and $M=I'M=MI$.

Let us show first that $\rho$ is surjective. Let $e$ be an
idempotent in $U$. Since $U/J\cong R/(I+I')$, there is an idempotent
$f\in R$ such that $\alpha (f)+J=e+J$. From \cite[Proposition
2.19]{vnrr} we get orthogonal decompositions $e=e_1+e_2$ and $\alpha
(f)=f_1+f_2$ such that $e_1\sim f_1$ and $e_2,f_2\in J$. Since
$\mon{\alpha}_{|\mon{I}}$ is an isomorphism from $\mon{I}$ onto
$\mon{J}$, we see that it is enough to show that $\mon{f_1}\in
\mon{\alpha}(\mon{R}) $. (Here $\mon{f_1}$ denotes the class of
$f_1$ in $\mon{U}$.)
 By the above observation we can find
idempotents $g_1\in I$ and $g_2\in I'$ such that
$g:=\begin{pmatrix}
g_1 & 0 \\
0 & g_2
\end{pmatrix}\in J$ satisfies $g\le \alpha (f)$ and $f_2\le g$. Now we have
an orthogonal decomposition
$$f_1= (\alpha (f)-g)+(g-f_2)$$
with $\alpha (f)-g\in \alpha (R)$ and $g-f_2\in J$. Consequently,
both $\mon{\alpha (f)-g}$ and $\mon{g-f_2}$ belong to the image of
$\mon{\alpha}$, and we get
$$\mon{e}=\mon{e_1}+\mon{e_2}=\mon{f_1}+\mon{e_2}=\mon{\alpha (f)-g}+\mon{g-f_2}+\mon{e_2}\in
\mon{\alpha}(\mon{R}),$$ as required.

Finally we prove the injectivity of $\rho$. Assume that $e$ and $f$
are idempotents in $R$ such that $\mon{\alpha (e)}=\mon{\alpha (f)}$
in $\mon{U}$. There are $x\in \alpha (e)U\alpha (f)$ and $y\in
\alpha (f)U\alpha (e)$ such that $xy=\alpha (e)$ and $yx=\alpha
(f)$. Write $x=\begin{pmatrix}
r_1+I' & n_1 \\
m_1 & r_1+I
\end{pmatrix}$ and $y=\begin{pmatrix}
r_2+I' & n_2 \\
m_2 & r_2+I
\end{pmatrix}$, where $r_1,r_2\in R$ and $n_i\in N$ and $m_i \in M$
for $i=1,2$. Applying the observation above to
$x_1=\begin{pmatrix}
0 & n_1 \\
m_1 & 0
\end{pmatrix}$ and $y_1=\begin{pmatrix}
0 & n_2 \\
m_2 & 0
\end{pmatrix}$, and to the
idempotent $\alpha (e) $, we get idempotents $g_1\in I$ and
$g_2\in I'$ such that $g:=\begin{pmatrix}
g_1 & 0 \\
0 & g_2
\end{pmatrix}\in J$ satisfies $g\le e$ and moreover $gx_1=x_1$ and $y_1g=y_1$.
But now $\alpha (e)=g+(\alpha(e)-g)$ and $\alpha (f)=ygx+y(\alpha
(e)-g)x$, with $ygx\sim g$ in $J$ and
\begin{align*}
y(\alpha(e)-g)x=&\begin{pmatrix}
(r_2+I')(e-g_1+I') & 0 \\
0 & (r_2+I)(e-g_2+I)
\end{pmatrix}\\
&\cdot \begin{pmatrix}
(e-g_1+I')(r_1+I') & 0 \\
0 & (e-g_2+I)(r_1+I)
\end{pmatrix}
\end{align*}
is equivalent to $\alpha(e)-g$ in $\alpha (R)$. Note that also
$ygx\in J$. Thus it remains to show that if
$\mbox{diag}(e_1,f_1)\sim \mbox{diag}(e_2,f_2)$ in $J$, then
$[\mon{e_1}+\mon{f_1}]=[\mon{e_2}+\mon{f_2}]$ in $\mon{R}/\sim$.
Using refinement (in $J$) we get orthogonal decompositions
$e_1=h_1+h_2$ and $f_1=h_1'+h_2'$ such that
$\mbox{diag}(h_1,h_1')\sim \mbox{diag}(e_2,0)$ and
$\mbox{diag}(h_2,h_2')\sim \mbox{diag}(0,f_2)$. So we have in
$\mon{R}/\sim$
\begin{align*}
[\mon{e_1}+\mon{f_1}] & =[\mon{h_1}+\mon{h_2}+\mon{h_1'}+\mon{h_2'}]\\
& =[\mon{h_1}+ \varphi (\mon{h_2})+\mon{h_1'}+\mon{h_2'}]\\ & =
[\mon{h_1}+\varphi ^{-1}(\mon{h_1'})+\varphi (\mon{h_2})+\mon{h_2'}]\\
&
  =[\mon{e_2}+\mon{f_2}].
\end{align*}
This completes the proof of the proposition.

\end{proof}

As an easy example to illustrate Proposition \ref{crownpushoutprop},
we consider the case $R=I\oplus I'$ with $I\cong I'\cong K$ as
rings. Then the ring $U$ produced in the proof of
\ref{crownpushoutprop} is just the ring $U=M_2(K)$ of $2\times 2$
matrices over $K$, and the map $\alpha \colon R\to U$ is just the
embedding along the diagonal.

\bigskip

We will also need a description of certain pullbacks of primitive
monoids

\begin{prop}\label{pullfgr}
Let $M_1$ and $M_2$ be primitive monoids. Let $N_i$ be an
order-ideal in $M_i$ such that $M_1/N_1\cong M_2/N_2\cong S$ and
let $P$ be the pullback of $M_1\longrightarrow S\longleftarrow
M_2$. Assume that $p\lhd q$ for all $p\in \mathbb P (N_i)$ and all
$q\in \mathbb P(M_i)\setminus \mathbb P (N_i)$, $i=1,2$. Then the
pullback $P$ is also a primitive monoid and $P$ has an order-ideal
$N\cong N_1\times N_2$ such that $P/N\cong S$. Moreover one has
$\mathbb P (P)=\mathbb P (S)\sqcup \mathbb P (N_1) \sqcup \mathbb
P (N_2)$, with the order relation $\lhd $ given by the relations
$\lhd$ of each monoid $S$, $N_1$ and $N_2$, and the further
relations $p\lhd q$ for all $p\in \mathbb P(N_1)\sqcup \mathbb
P(N_2)$ and $q\in \mathbb P (S)$.
\end{prop}

\begin{proof}
By Lemma \ref{lem:quotient}, $S$ is a primitive monoid and
$\mathbb P(S)$ is $\lhd$-isomorphic to $\mathbb P(M_i)\setminus
\mathbb P (N_i)$ for $i=1,2$. Let $Q$ be the primitive monoid
determined by the set of primes $\mathbb P (Q)=\mathbb P (S)\sqcup
\mathbb P (N_1) \sqcup \mathbb P (N_2)$, with the relations
inherited by each monoid and $p\lhd q$ for all $p\in \mathbb
P(N_1)\sqcup \mathbb P(N_2)$ and $q\in \mathbb P (S)$. The result
that $P\cong Q$ follows easily from the existence of the
commutative diagram

$$
\begin{CD}
& & & & 0 & & 0\\
& & & & @VVV   @VVV   \\
& & & & N_1 @>=>>  N_1 \\
& & & & @VVV @VVV \\
0 @>>> N_2 @>>> Q @>>> M_1 @>>> 0 \\
& & @V=VV  @VVV  @VVV  \\
0 @>>> N_2  @>>> M_2 @>>> S @>>> 0\\
& & & & @VVV   @VVV   \\
& & & & 0 &   & 0
\end{CD}
$$

\medskip

\noindent where here a short exact sequence such as $0\to N\to
M\to T\to 0$ means that $N$ is an order-ideal of $M$ such that
$M/N\cong T$.
\end{proof}

\section{The building blocks}
\label{sect:building}

In this section we recall some notation and results from \cite{AB2}
and \cite{AB3} that we will need in the proof of our main result.
For the sake of clarity, we will give a direct proof (modulo some
basic results in \cite{AB2} and \cite{AB3}) of the result which we
need in the present paper.

In the following, $\cb$ will denote a field and $E=(E^0,E^1,r,s)$ a
finite quiver (oriented graph). Here $s(e)$ is the {\em source
vertex} of the arrow $e$, and $r(e)$ is the {\em range vertex} of
$e$. A {\em path} in $E$ is either an ordered sequence of arrows
$\alpha=e_1\dotsb e_n$ with $r(e_t)=s(e_{t+1})$ for $1\leqslant
t<n$, or a path of length $0$ corresponding to a vertex $v\in E^0$.
The paths $v$ are called trivial paths, and we have $r(v)=s(v)=v$. A
non-trivial path $\alpha=e_1\dotsb e_n$ has length $n$ and we define
$s(\alpha)=s(e_1)$ and $r(\alpha)=r(e_n)$. We will denote the length
of a path $\alpha$ by $|\alpha|$, the set of all paths of length $n$
by $E^n$, and the set of all paths by $E^*$.

We define a relation $\ge$ on $E^0$ by setting $v\ge w$ if there is
a path $\mu\in E^*$ with $s(\mu)=v$ and $r(\mu)=w$. A subset $H$ of
$E^0$ is called \emph{hereditary} if $v\ge w$ and $v\in H$ imply
$w\in H$. A set is \emph{saturated} if every vertex which feeds into
$H$ and only into $H$ is again in $H$, that is, if $s^{-1}(v)\neq
\emptyset$ and $r(s^{-1}(v))\subseteq H$ imply $v\in H$.

In this paper we will only need a special case of the general
construction. For any positive integer $r\ge 0$, we consider the
graph $E_r$ consisting of $r+1$ vertices $v_0,\dots ,v_r$ and for
each $1\le i\le r$, two arrows $a_i,b_i$ such that
$s(a_i)=r(a_i)=s(b_i)=v_i$ and $r(b_i)=v_{i-1}$. A picture of $E_3$
is shown in Figure \ref{Fig:QuiverE3}. Observe that $(E_r^0, \le)$
is a chain of length $r$. The monoid $M(E_r)$ of \cite{AMP} agrees
with the the monoid $M(E_r^0)$ associated with the poset $E_r^0$.

\begin{figure}[htb]
 \[
 {
 \xymatrix{
 & v_{3}\ar@(ul,ur)\ar[r] & v_{2} \ar@(ul,ur)\ar[r] & v_{1} \ar@(ul,ur)\ar[r] & v_{0}}
 }
 \]
\caption{The quiver $E_3$.} \label{Fig:QuiverE3}
\end{figure}
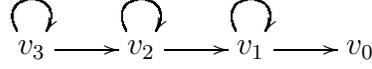

Observe that there is a unique maximal chain of hereditary saturated
subsets of $E_r^0$
$$\mathbf{H}_r : \{ v_0\}\subset \{ v_0,v_1 \}\subset \cdots \subset E_r^0.$$

Let us recall the construction from \cite{AB2} of the regular
algebra $Q_K(E)$ of a quiver $E$. Unfortunately the arrows are taken
in the present paper in the reverse sense as in \cite{AB2}. So we
recall the basic features of the regular algebra $Q_K(E)$ in terms
of the notation used here. The algebra $Q_K(E)$ fits into the
following commutative diagram of injective algebra morphisms:

\begin{equation}\notag \begin{CD}
K^{|E^0|} @>>> P_K(E) @>{\iota _{\Sigma}}>> P^{\text{rat}}_K(E) @>>> P_K((E))\\
@VVV @V{\iota _{\Sigma_1}}VV @V{\iota _{\Sigma _1}}VV @V{\iota _{\Sigma _1}}VV\\
P_K(\ol{E}) @>>> L_K(E) @>{\iota _{\Sigma}}>> Q_K(E) @>>> U_K(E)
\end{CD} \notag \end{equation}

Here $P_K(E)$ is the path $K$-algebra of $E$, $\ol{E}$ denotes the
inverse quiver of $E$, that is, the quiver obtained by changing the
orientation of all the arrows in $E$, $P_K((E))$ is the algebra of
formal power series on $E$, and $P^{\text{rat}}_K(E)$ is the algebra
of rational series, which is by definition the division closure of
$P_K(E)$ in $P((E))$ (which agrees with the rational closure
\cite[Observation 1.18]{AB2}). The maps $\iota_{\Sigma}$ and $\iota
_{\Sigma_1}$ indicate universal localizations with respect to the
sets $\Sigma$ and $\Sigma _1$ respectively. Here $\Sigma$ is the set
of all square matrices over $P_K(E)$ that are sent to invertible
matrices by the augmentation map $\epsilon \colon P_K(E)\to
K^{|E^0|}$. By \cite[Theorem 1.20]{AB2}, the algebra
$P^{\text{rat}}_K(E)$ coincides with the universal localization
$P_K(E)\Sigma ^{-1}$. For $v\in E^0$ with $s^{-1}(v)\ne \emptyset$,
write $s^{-1}(v)=\{e^v_1,\dots ,e^v_{n_v}\}$. The set
$\Sigma_1=\{\mu_v\mid v\in E^0,\,s^{-1}(v)\neq \emptyset\}$ is the
set of morphisms between finitely generated projective left
$P_K(E)$-modules defined by
 \begin{align*}
  \mu_v\colon P_K(E)v&\longrightarrow \bigoplus_{i=1}^{n_v}P_K(E)r(e^v_i)\\
  r&\longmapsto\left(re^v_1,\dotsc,re^v_{n_v}\right)
 \end{align*}
for any $v\in E^0$ such that $s^{-1}(v)\neq\emptyset$. By a slight
abuse of notation, we use also $\mu _v$ to denote the corresponding
maps between finitely generated projective left
$P^{\text{rat}}_K(E)$-modules and $P_K((E))$-modules respectively.

The following relations hold in $Q_K(E)$:

(1) $vv'=\delta _{v,v'}v$ for all $v,v'\in E^0$.

(2) $s(e)e=er(e)=e$ for all $e\in E^1$.

(3) $r(e)\ol{e}=\ol{e}s(e)=\ol{e}$ for all $e\in E^1$.

(4) $\ol{e}e'=\delta _{e,e'}r(e)$ for all $e,e'\in E^1$.

(5) $v=\sum _{\{ e\in E^1\mid s(e)=v \}}e\ol{e}$ for every $v\in
E^0$ that emits edges.

The Leavitt path algebra $L_K(E)=P_K(E)\Sigma_1^{-1}$ is the algebra
generated by $\{v\mid v\in E^0\}\cup \{e,\ol{e}\mid e\in E ^1\}$
subject to the relations (1)--(5) above. By \cite[Theorem 4.2]{AB2},
the algebra $Q_K(E)$ is a von Neumann regular hereditary ring and
$Q_K(E)=P_K(E)(\Sigma \cup \Sigma_1)^{-1}$. Here the set $\Sigma$
can be clearly replaced with the set of all square matrices of the
form $I_n+B$ with $B\in M_n(P_K(E))$ satisfying $\epsilon (B)=0$,
for all $n\ge 1$, since the map $\epsilon \colon P_K(E)\to
K^{|E^0|}$ is a split surjection. The graph monoid $M(E)$ of $E$ is
defined as the quotient monoid of $F=F_E$, the free abelian monoid
with basis $E^0$, modulo the congruence generated by the relations
$$v=\sum _{\{ e\in E^1\mid s(e)=v \}} r(e) $$
for every vertex $v\in E^0$ such that $s^{-1}(v)\ne \emptyset$. It
was proved in \cite[Theorem 3.5]{AMP} that the natural map
$M(E)\to\mon{L_K(E)}$ is an isomorphism, and in \cite[Theorem
4.2]{AB2} that the map $\mon{L_K(E)}\to \mon{Q_K(E)}$ induced by the
inclusion $L_K(E)\to Q_K(E)$ is also an isomorphism.

The structure of the lattice of ideals of $Q_K(E)$ can be neatly
computed from the graph. Let $H$ be a hereditary saturated subset of
$E^0$. Define the graph $E/H$ by $(E/H)^0=E^0\setminus H$ and
$(E/H)^1=\{e\in E^1: r(e)\notin H\}$, with the functions $r$ and $s$
inherited from $E$. We also define $E_H$ as the restriction of the
graph $E$ to $H$, that is $(E_H) ^0=H$ and $(E_H)^1=\{e\in E^1:
s(e)\in H\}$. For $Y\subseteq E^0$ set $p_Y=\sum _{v\in Y} v$.

\begin{prop}
\label{quotient}

(a) The ideals of $Q_K(E)$ are in one-to-one correspondence with the
order-ideals of $M(E)$ and consequently with the hereditary and
saturated subsets of $E$.

(b) If $H$ is a hereditary saturated subset of $E$, then
$Q_K(E)/I_K(H)\cong Q_K(E/H)$, where $I_K(H)$ is the ideal of
$Q_K(E)$ generated by the idempotent $p_H=\sum _{v\in H} v$.

(c) Let $H$ be a hereditary subset of $E^0$. Then the following
properties hold:
\begin{enumerate}
\item $P_K(E_H)=p_HP_K(E)=p_HP_K(E)p_H$,
\item $P_K((E_H))=p_HP_K((E))=p_HP_K((E))p_H$,
\item $P_K^{\text{rat}}(E_H)=p_HP_K^{\text{rat}}(E)=p_HP_K^{\text{rat}}(E)p_H$,
\item $Q_K(E_H)\cong p_HQ_K(E)p_H$.
\end{enumerate}
\end{prop}

\begin{proof}
(a) By \cite[Theorem 4.2]{AB2} we have a monoid isomorphism
$\mon{Q_K(E)}\cong M(E)$. Since $Q_K(E)$ is von Neumann regular, we
have a lattice isomorphism $\mathcal L (Q_K(E))\cong \mathcal
L(M(E))$, see Proposition \ref{prop:wellknownideals}. Now by
\cite[Proposition 5.2]{AMP} there is a lattice isomorphism $\mathcal
L(M(E))\cong \mathcal H$, where $\mathcal H$ is the lattice of
hereditary saturated subsets of $E^0$.

(b), (c)  See \cite{AB3}.
\end{proof}

We are ready to define the basic building blocks to apply the
diagram constructions.

\begin{defi}\label{buildblock}
Let $\mathbf K_r: \,K_0\subseteq K_1\subseteq \cdots \subseteq K_r$
be a chain of fields. Let $E_r$ be the quiver defined above and let
$\mathbf H_r :\, H_0=\{v_0\}\subset H_1=\{v_0,v_1\}\subset \cdots
\subset H_r=E_r^0$ be the unique maximal chain of hereditary
saturated subsets of $E_r^0$. Build rings $R_i$ $i=0,1,\dots ,r$
inductively as follows:

{\rm (1)} $R_0=Q_{K_r}(E_{H_0})\cong K_r$.

{\rm (2)} $R_i=Q_{K_{r-i}}(E_{H_i})+
Q_{K_{r-i}}(E_{H_i})p_{H_{i-1}}R_{i-1}p_{H_{i-1}}Q_{K_{r-i}}(E_{H_i}) $ for $1\le i\le r$.

Each $R_i$ is a unital $K_{r-i}$-algebra with unit $p_{H_i}$ and we
have $Q_{K_{r-i}}(E_{H_i})\subseteq R_i\subseteq
Q_{K_r}(E_{H_i})=p_{H_i}Q_{K_r}(E_r)p_{H_i}\subseteq Q_{K_r}(E_r)$.
Put $Q_{\mathbf{K}_r}(E_r;\mathbf{H}_r)=R_r$.
\end{defi}

\begin{theor}
\label{hereditarymixed} With the above notation, we have that
$Q_{\mathbf{K}_r}(E_r;\mathbf{H}_r)$ is a von Neumann regular ring
and the natural map $M(E_r)\to
\mon{Q_{\mathbf{K}_r}(E_r;\mathbf{H}_r)}$ is an isomorphism.
\end{theor}

\begin{proof}
We proceed by induction on $r$. For $r=0$, we have
$Q_{\mathbf{K}_0}(E_0;\mathbf{H}_0)\cong K_0$ so the result
trivially holds.

Assume that $r>0$ and that the result holds for $r-1$. Let $I_{r-1}$
be the ideal of $Q_{\mathbf{K}_r}(E_r;\mathbf{H}_r)$ generated by
$p_{H_{r-1}}=\sum _{i=0}^{r-1} v_i$. By Proposition
\ref{quotient}(b), we have
\begin{equation}
\label{quoffield} Q_{\mathbf{K}_r}(E_r;\mathbf{H}_r)/I_{r-1}\cong
Q_{K_0}(E_r/H_{r-1})\cong K_0(z) ,
\end{equation} the rational function field in
one variable over the field $K_0$. Thus
$Q_{\mathbf{K}_r}(E_r;\mathbf{H}_r)/I_{r-1}$ is von Neumann regular
and $\mon{Q_{\mathbf{K}_r}(E_r;\mathbf{H}_r)/I_{r-1}}\cong \mathbb
Z^+$.

On the other hand we have (\cite{AB3})
$$p_{H_{r-1}}Q_{\mathbf{K}_r}(E_r;\mathbf{H}_r)p_{H_{r-1}}\cong Q_{\mathbf{K}^{r-1}}(E_{H_{r-1}};
\mathbf{H}_{r-1}) ,$$ where
$$\mathbf{K}^{r-1} : K_{1}\subseteq K_{2} \subseteq \cdots \subseteq K_{r} $$
and
$$\mathbf{H}_{r-1}:  H_0\subset H_1
\subset \cdots \subset H_{r-1} .$$

Thus, by induction hypothesis,
$p_{H_{r-1}}Q_{\mathbf{K}_r}(E_r;\mathbf{H}_r)p_{H_{r-1}}$ is von
Neumann regular and the natural map $M(E_{r-1})\to
\mon{p_{H_{r-1}}Q_{\mathbf{K}_r}(E_r;\mathbf{H}_r)p_{H_{r-1}}}$ is
an isomorphism. Now we want to check that $I_{r-1}$ has local units.
Note that
$$I_{r-1}=(Q_{K_0}(E_r)p_{H_{r-1}})(p_{H_{r-1}}R_{r-1}p_{H_{r-1}})(p_{H_{r-1}}Q_{K_0}(E_r)).$$
Given $x\in I_{r-1}$, we can write
$$x=\sum _{i=1}^n a_iz_ib_i, \quad \text{with } a_i\in
Q_{K_0}(E_r)p_{H_{r-1}},\,\, z_i\in p_{H_{r-1}}R_{r-1}p_{H_{r-1}}\,
\text{ and }\, b_i\in p_{H_{r-1}}Q_{K_0}(E_r) .$$ Since
$Q_{K_0}(E_r)p_{H_{r-1}}Q_{K_0}(E_r)$ is von Neumann regular
(\cite[Theorem 4.2]{AB2}), it has local units (see, e.g.
\cite[Example 1]{AM}) so there is an idempotent $e\in
Q_{K_0}(E_r)p_{H_{r-1}}Q_{K_0}(E_r)\subseteq I$ such that $ea_i=a_i$
and $b_i=b_ie$ for all $i$. It follows that $ex=x=xe$, as desired.

Since $I_{r-1}$ has local units and $I_{r-1}$ is Morita equivalent
to the regular ring \linebreak
$p_{H_{r-1}}Q_{\mathbf{K}_r}(E_r;\mathbf{H}_r)p_{H_{r-1}}$, it
follows from \cite[Proposition 3.1]{AM} that $I_{r-1}$ is von
Neumann regular. By using this and (\ref{quoffield}), it follows
from \cite[Lemma 1.3]{vnrr} that
$Q_{\mathbf{K}_r}(E_r;\mathbf{H}_r)$ is von Neumann regular.
Moreover since it is an extension of two strongly separative regular
rings, we get that $Q_{\mathbf{K}_r}(E_r;\mathbf{H}_r)$ is also
strongly separative (\cite[Theorem 5.2]{AGOP}).

Observe that by the Morita invariance of the functor $\mon{-}$ (see
\cite[Corollary 5.6]{GoodLeavitt}) it follows that the natural map
$M(E_{r-1})\to \mon{I_{r-1}}$ is an isomorphism.

We now want to compute the monoid
$\mon{Q_{\mathbf{K}_r}(E_r;\mathbf{H}_r)}$. Write
$Q:=Q_{\mathbf{K}_r}(E_r;\mathbf{H}_r)$. We have algebra embeddings
$Q_{K_0}(E_r)\to Q\to Q_{K_r}(E_r)$ which induce monoid
homomorphisms
$$M(E_r)\cong \mon{Q_{K_0}(E_r)}\to \mon{Q}\to \mon{Q_{K_r}(E_r)}\cong M(E_r).$$
Observe that the composition above is the identity, so it follows
that the map $$\mon{Q_{K_0}(E_r)}\to \mon{Q}$$ is injective. In
order to see that it is surjective we have to prove that every
idempotent in $Q$ is equivalent to a direct sum of basic idempotents
in $E_r^0$. Take an idempotent $e$ in $Q$. If $e\in I_{r-1}$ then
the result follows from the isomorphism $M(E_{r-1})\cong
\mon{I_{r-1}}$. If $e\notin I_{r-1}$, then $[e+I_{r-1}]=[1]$ by
(\ref{quoffield}), and thus we get $e\oplus f_1\sim 1\oplus f_2$ for
some idempotents $f_1,f_2\in I_{r-1}$. Now since $1\oplus f\sim 1$
for every idempotent $f\in I_{r-1}$, we get $$e\oplus f_1\sim
1\oplus f_2\sim 1\sim 1\oplus f_1,$$ so that $e\oplus f_1\sim
1\oplus f_1$. Since $f_1\le 1$ and $Q$ is strongly separative we get
that $e\sim 1$, as desired.

We conclude that the natural map $M(E_r)\to \mon{Q}$ is an
isomorphism, as wanted.
\end{proof}

\section{The proof of the main Theorem}
\label{sect:theproof}

This section is devoted to the proof of our main result, Theorem
\ref{main}. Let $M$ be a finitely generated primitive monoid with
all primes free, and let $\mathbb P=\mathbb P (M)$ be the poset of
primes of $M$. The proof of Theorem \ref{main} is naturally divided
into two big steps, as follows:

\medskip

Step 1: For a maximal element $p$ of $\mathbb P$, the poset $\mathbb
P\dnw p=\{q\in \mathbb P: q\le p\}$ has $p$ as a greatest element,
and we build in Proposition \ref{posets} a poset $\mathbb F (p)$,
with greatest element $p$, such that, for each $t\in \mathbb F (p)$,
the interval $[t,p]$ is a chain, together with an order-preserving
surjective map $\Psi_p\colon \mathbb F (p)\to \mathbb P \dnw p$. The
poset $\mathbb F (p)$ is built up from the different {\it maximal
chains} of $\mathbb P\dnw p$, and Proposition \ref{pullfgr} tells us
that $M(\mathbb P\dnw p)=M(\mathbb P)\mid p$ is built up from the
corresponding monoids $M(S)$ by a finite sequence of pullback
diagrams. Finally to each maximal chain $S$ we associate a basic
building block $Q(S)$ of the sort considered in Section
\ref{sect:building}, in such a way that, using Theorem
\ref{pullbacktheorem}, we are able to prove that the same sequence
of pullback diagrams applied now to the $K$-algebras $Q(S)$ leads us
to the algebra $Q_K(\mathbb F (p))$ of the poset $\mathbb F(p)$
(Definition \ref{defi: Q(M)}), so that Theorem 2.3 is proved for the
posets $\mathbb F (p)$. This is achieved in Theorem \ref{firstcase}.
This step uses Sections \ref{QQQ}, \ref{sect:pullbacks} and
\ref{sect:building}, as well as Proposition \ref{pullfgr}.

\medskip

Step 2: Let $\Psi_p\colon \mathbb F (p)\to \mathbb P\dnw p$ be the
map of posets described in Step 1. This map can be extended to a
surjective monoid homomorphism $\Psi_p\colon M(\mathbb F(p))\to
M(\mathbb P)\mid p$. We show in Proposition \ref{primitivepushout}
that $M(\mathbb P)\mid p$ is obtained from $M(\mathbb F (p))$ by a
finite sequence of crowned pushouts of the form considered in Lemma
\ref{structVcrownpushout}. Moreover $M(\mathbb P)$ is also obtained
from $\prod _{p\in \text{Max}(\mathbb P)} M(\mathbb P)\mid p$ by a
sequence of crowned pushouts as in \ref{structVcrownpushout}. Now
Proposition \ref{crownpushoutprop} gives a way to construct, from a
given von Neumann regular ring $R$ such that $\mon{R}\cong M$, and
from a suitable crowned pushout $M'$ of $M$, a von Neumann regular
ring $U$ such that $\mon{U}\cong M'$. Since we already know that
Theorem \ref{main} holds for the posets $\mathbb F (p)$ (Step 1),
the result in the general case follows by an inductive argument from
Proposition \ref{crownpushoutprop}, once we are able to identify
(modulo Morita-equivalence) the algebra $U$ for $R=Q_K(M)$,  with
the corresponding algebra $Q_K(M')$, where $M$ is one of the monoids
appearing in the sequence of crowned pushouts leading from the
different monoids $M(\mathbb F (p))$, for $p\in \text{Max}(\mathbb
P)$, to $M(\mathbb P)$, and $M'$ is obtained from $M$ by a crowned
pushout construction as shown in the proof of
\ref{primitivepushout}. This identification is done in Theorem 6.6.
Step 2 uses Sections \ref{QQQ} and \ref{sect:pushout}.

\bigskip

Let $M$ be a finitely generated primitive monoid such that all the
primes of $M$ are free. The {\it height} of a prime $p\in \mathbb P
(M)$ is the length $r$ of a maximal chain of primes $p_0< p_1<
\cdots < p_r=p$. The height of $M$ is the maximum of the heights of
its prime elements. Similar definitions apply to (elements of) a
finite poset.

We first state a purely order-theoretic result.

\begin{prop}\label{posets}
Let $\mathbb P(p)$ be a finite poset with a greatest element $p$,
and assume that the height of $p$ is $r$. Let $\mathcal S ^0(p)$ be
the set of all maximal chains in $\mathbb P(p)$ of the form
$p_0<p_1<\cdots <p_s=p$, and let $\Psi_0 \colon \bigsqcup_{S\in
\mathcal S^0(p)}S\to \mathbb P (p)$ be the natural surjective
identification map. Then there is a sequence of sets $\mathcal
S^1(p), \mathcal S^2(p),\dots ,\mathcal S^r(p)$ such that:
\begin{enumerate}[(i)]
\item $\mathcal S^{i}(p)$ consists of partially ordered sets $T$
with maximum element $p$ such that for each $t\in T$ the set
$[t,p]=\{x\in T: t\le x\}$ is a chain. Moreover $T$ contains a
chain $p_i< \cdots < p_r=p$ and every element in $T$ not in this
chain is below $p_i$, that is $T=\{x\in T: x<p_i\}\sqcup
\{p_i,\dots ,p_r\}$.
\item For each $i$ there is a surjective identification order-preserving map
$$\Psi_i \colon \bigsqcup_{S\in \mathcal
S^i(p)}S\to \mathbb P(p) .$$
\item $\mathcal S^r(p)$ is a singleton $\{\mathbb F (p)\}$, so there
is a surjective identification order-preserving map  $$\Psi
_r\colon \mathbb F(p)\to \mathbb P(p) .$$
\end{enumerate}
\end{prop}

\begin{proof}
In order to make the process clear, we will start by constructing
$\mathcal S^1(p)$. Let us define an equivalence relation on the
set $\mathcal S^0(p)$ by setting $S\sim S'$ iff either $S=S'$ or
$S\ne S'$ and $p_i=p_i'$ for $i=1,\dots ,r$. In the latter case,
we are assuming that both $S$ and $S'$ have maximal length $r$.
Each chain in $\mathcal S^0(p)$ of length $<r$ forms a singleton
class with respect to $\sim$. Now $\mathcal S^1(p)$ contains all
the chains $S$ in a singleton class and one new partially ordered
set for each equivalence class with more than one element. We
construct this element for a given class $\{ S_1,\dots
,S_{\alpha}\}$. Let $p_{0,1},\dots ,p_{0,\alpha}$ be the minimal
elements in each one of the chains $S_1, \dots ,S_{\alpha}$, and
let $p_1< \cdots < p_r$ be the common part of the chains. Then the
set $T$ corresponding to this class has elements $p_{0,1}, \dots
,p_{0,\alpha},p_1,\dots ,p_r$ and the order relation determined by
$p_{0,i}< p_1< \cdots < p_r$ for all $i$. There is an obvious
surjective identification order-preserving map
$$\Psi_1 \colon \bigsqcup_{T\in \mathcal
S^1(p)}T\to \mathbb P(p) .$$

Now the construction of $\mathcal S ^i(p)$ from $\mathcal
S^{i-1}(p)$ is similar to the first step. By induction hypothesis,
the set $\mathcal S^{i-1}(p)$ consists of partially ordered sets $T$
with maximum element $p$ such that for each $t\in T$ the set
$[t,p]=\{x\in T: t\le x\}$ is a chain. Moreover $T$ contains a chain
$p_{i-1}< p_i< \cdots < p_r=p$ (with $p_{\ell-1}\in \rL(\mathbb
P(p), p_{\ell})$ for $\ell =i,\dots ,r$) and $T=\{x\in T:x<
p_{i-1}\}\sqcup \{p_{i-1},\dots ,p_r\}$. We also have surjective
identification maps
$$\Psi _{i-1}\colon \bigsqcup_{T\in \mathcal S^{i-1}(p)}T\to \mathbb
P (p) .$$

Now define an equivalence relation on $\mathcal S^{i-1}(p)$ by
$T\sim T'$ iff  $p_j=p'_j$ for $j=i,i+1,\dots ,r$, where $T$
contains the chain $p_{i-1}< p_{i}< \cdots < p_r=p$ and $T'$
contains the chain $p_{i-1}'< p_{i}'< \cdots < p_r'=p$, with
$p_{\ell-1}\in \rL(\mathbb P(p), p_{\ell})$ and $p'_{\ell-1}\in
\rL(\mathbb P(p), p'_{\ell})$ for $\ell =i,\dots ,r$.

The set $\mathcal S^i(p)$ contains all the posets $T$ in a
singleton class of $S^{i-1}(p)$ and one new partially ordered set
for each equivalence class with more than one element. For a given
class $\{ T_1,\dots ,T_k\}$, take
$$W=\{p_i,p_{i+1},\dots ,p_r\}\sqcup \bigsqcup _{j=1}^{k}\{t\in T_j:t< p_{i}\},$$
with the obvious order relation. The restriction of $\Psi _{i-1}$
to $\bigsqcup _{j=1}^{k} T_j$ factors through $W$, so we obtain a
surjective order-preserving map $\Psi_i \colon \bigsqcup_{W\in
\mathcal S^i(p)}W\to \mathbb P(p) $. For $i=r$, we get a unique
poset $\mathbb F (p)$.
\end{proof}

Note that the map $\Psi _r\colon \mathbb F (p)\to \mathbb P (p)$
is an isomorphism if and only if $\mathbb P (p)$ satisfies that
$[t,p]$ is a chain for every $t\in\mathbb P(p)$. The first step of
the proof of Theorem \ref{main} consists in showing it in the
particular case where the poset $\mathbb P (M)$ has a maximum
element $p$ and $[t,p]$ is a chain for every $t\in \mathbb P (M)$
(Theorem \ref{firstcase}).

 For any positive integer $r\ge 1$, we consider the graph $E_r$
consisting of $r+1$ vertices $v_0,\dots ,v_r$ and for each $1\le
i\le r$, two arrows $a_i,b_i$ such that $s(a_i)=r(a_i)=s(b_i)=v_i$
and $r(b_i)=v_{i-1}$. Let $M_r$ be the graph monoid of $E_r$. Then
$\mathbb P (M_r)$ is a chain of length $r$, that is, it consists
of $r+1$ primes $p_0,p_1,\dots ,p_r$ corresponding to
$v_0,v_1,\dots ,v_r$ respectively, such that
$$p_0< p_1< \cdots < p_r.$$
Now we consider a construction as in Section \ref{sect:building},
with respect to the following data. Let $K(t_1,t_2,\dots )$ be an
infinite purely transcendental extension of a field $K$. Select
positive integers $1\le k_1\le k_2\le \cdots \le k_r$ and set
$$\mathbf{K}_r : K(t_{k_r},t_{k_r+1},\dots )\subseteq
K(t_{k_{r-1}},t_{k_{r-1}+1},\dots )\subseteq \cdots \subseteq
K(t_1,t_2, \dots )=L .$$ We also consider the unique maximal chain
of hereditary saturated subsets of $E_r^0$
$$\mathbf{H}_r : \{ v_0\}\subset \{ v_0,v_1 \}\subset \cdots \subset E_r^0.$$
By Theorem \ref{hereditarymixed}, the ring
$Q_{\mathbf{K}_r}(E_r;\mathbf{H}_r)$ is von Neumann regular and
$\mathcal V (Q_{\mathbf{K}_r}(E_r;\mathbf{H}_r))=M_r$. First of all
we want to compare this construction with the construction in
Definition \ref{defi: Q(M)}. For this we consider a slight
variation, denoted by $Q(M_r,\sigma(k_1,\dots ,k_r))$, of the
definition in \ref{defi: Q(M)}, where the $K$-endomorphisms
$\sigma^{p_i}\colon L\to L$ used in (\ref{eq:A.8}) are replaced with
the ones defined by the rule
$$\sigma^{p_i}(t_j)=t_{j+k_i-k_{i-1}}.$$
Observe that $Q(M_r,\sigma(k_1,\dots ,k_r))$ has the same essential
properties as $Q_K(M_r)$ (cf. Section \ref{QQQ}).

\begin{prop}
\label{comparisonQMr} With the above notation we have an isomorphism
of $K$-algebras $$\gamma\colon Q(M_r,\sigma(k_1,k_2,\dots
,k_r))\longrightarrow Q_{\mathbf{K}_r}(E_r;\mathbf{H}_r)$$ such that
$\gamma (e(p_i))=v_i$ for $i\ge 0$, $\gamma (\alpha
_{p_i,p_{i-1}})=a_i$, $\gamma (\beta_{p_i,p_{i-1}})=b_i$ and $\gamma
(e(p_i)t_j)=v_it_{j+k_i-1}$ for $i\ge 1$.
\end{prop}

\begin{proof} Write $\alpha _i=\alpha _{p_i,p_{i-1}}$ and $\beta
_i=\beta_{p_i,p_{i-1}}$. Observe that $Q(M_r,\sigma(k_1,k_2,\dots
,k_r))=\mathcal A (\Sigma\cup \Sigma_1)^{-1}$ (see Definition
\ref{defi: Q(M)}).

Define a $K$-algebra map $\widetilde{\gamma}\colon \mathcal A\to
Q_{\mathbf{K}_r}(E_r;\mathbf{H}_r)$ by the rules:

\begin{align*}
\widetilde{\gamma}(e(p_i)) & =v_i,\quad \widetilde{\gamma}(\alpha
_i)=a_i, \quad \widetilde{\gamma}(\beta _i)=b_i,\\
\widetilde{\gamma}(e(p_i,p_{i-1})) & =b_i\ol{b}_i,\quad
\text{and}\quad \widetilde{\gamma}(e(p_i)t_j)=v_it_{j+k_i-1}
\end{align*}
 Since relations (\ref{eq:A.3}), (\ref{eq:A.8}) and
(\ref{eq:A.10}) are respected by $\widetilde{\gamma}$, the map
$\widetilde{\gamma}$ is well-defined. Clearly the map
$\widetilde{\gamma}$ is $(\Sigma \cup \Sigma _1)$-inverting, so
that we have a well-defined $K$-algebra morphism
$$\gamma \colon Q(M_r,\sigma(k_1,\dots ,k_r))\to
Q_{\mathbf{K}_r}(E_r;\mathbf{H}_r).$$

In order to show that $\gamma $ is surjective, it is enough to show,
in view of Definition \ref{buildblock}, that
$Q_{K_{r-i}}(E_{H_i})\subseteq \text{Im}(\gamma)$. For this, we will
define a map $\tau_i\colon Q_{K_{r-i}}(E_{H_i})\to
Q(M_r,\sigma(k_1,k_2,\dots ,k_r))$  such that $\gamma\tau_i
=\text{id}_{Q_{K_{r-i}}(E_{H_i})}$. Recall that the algebra
$Q_{K_{r-i}}(E_{H_i})$ is the universal localization of the path
algebra $P_{K_{r-i}}(E_{H_i})$ with respect to $\Sigma
(\epsilon_i)\cup \Xi_i$, where $\epsilon_i \colon
P_{K_{r-i}}(E_{H_i})\to  K_{r-i}^{i+1}$ is the augmentation map, and
$\Xi _i=\{\mu_{v_j}:j=1,\dots ,i\}$, with
\begin{align*}
\mu_{v_j}\colon P_{K_{r-i}}(E_{H_i})v_j&\longrightarrow
P_{K_{r-i}}(E_{H_i})v_j\oplus P_{K_{r-i}}(E_{H_i})v_{j-1}.\\
r&\longmapsto\left( ra_j, rb_j\right)
\end{align*}
Define a map $\widetilde{\tau }\colon P_{K_{r-i}}(E_{H_i})\to (\sum
_{j=0}^i e(p_i))Q(M_r,\sigma(k_1,\dots ,k_r))(\sum _{j=0}^i e(p_i))$
by the rule
$$\widetilde{\tau}(v_j)=e(p_j),\quad \widetilde{\tau}(a_j)=\alpha
_j,\quad \widetilde{\tau}(b_j)=\beta _j,\quad
\widetilde{\tau}(v_jt_{\ell})=e(p_j)t_{\ell-k_j+1}.$$ The map is
clearly $\Xi_i$-inverting. We have to show that $\widetilde{\tau}$
is $\Sigma (\epsilon_i)$-inverting. Set $P=P_{K_{r-i}}(E_{H_i})$,
and take any matrix $A\in M_n(P)$ such that $\epsilon_i (A)=0$.
Observe that we can write
$$A=A_i+B_i+A_{i-1}+B_{i-1}+\cdots + A_1+B_1,$$
with $A_j\in v_jM_n(P)v_j$ and $B_j\in v_jM_n(P)(v_{j-1}+\cdots
+v_0)$, and with $\epsilon _i(A_j)=\epsilon _i(B_{\ell})=0$ for all
$j,\ell$. Since
$$(I_n-A)^{-1}=(I_n-A_i)^{-1}(I_n-B_i)^{-1}(I_n-A_{i-1})^{-1}\cdots (I_n-A_1)^{-1}
(I_n-B_1)^{-1}$$ in any ring in which all the terms in the RHS are
invertible, it suffices to show that the terms
$I_n-\widetilde{\tau}(A_j)$ and $I_n-\widetilde{\tau}(B_j)$ are
invertible in $Q(M_r,\sigma(k_1,\dots ,k_r))$ for all $j\le i$.
Since $B_j^2=0$, the matrices $I_n-B_j$ are obviously invertible. On
the other hand, since, for $j\ge 1$, we have $v_jPv_j\cong
K_{r-i}[z]$, we get that
$$\text{det}(I_n-\widetilde{\tau}(A_j))=1-f_j(\alpha
_j)$$ where $f_j\in L[z]$  satisfies $f_j(0)= 0$. It follows that
$e(p_j)-f_j(\alpha _j)\in \Sigma (p_j)\subseteq \Sigma $ (see
(\ref{eq:A.6}) and (\ref{eq:A.7})) and thus $I_n-\widetilde{\tau
}(A_j)$ is invertible in $Q(M_r,\sigma(k_1,\dots ,k_r))$ for all
$j$. It follows that there is a unique extension of
$\widetilde{\tau}$ to a $K_{r-i}$-algebra map
$$\tau_i \colon Q_{K_{r-i}}(E_{H_i})\longrightarrow   (\sum
_{j=0}^i e(p_i))Q(M_r,\sigma(k_1,\dots ,k_r))(\sum _{j=0}^i
e(p_i)).$$ Since the composition $\gamma\tau _i$ is the identity on
$P_{K_{r-i}}(E_{H_i})$, it must be the identity also on
$Q_{K_{r-i}}(E_{H_i})$.

We have shown that the map $\gamma $ is surjective. The injectivity
of $\gamma$ follows easily from Lemma \ref{lem:injectivity}.
\end{proof}

We are now ready to show the following particular case of our main
result:

\begin{theor}\label{firstcase} Let $M$ be a finitely generated primitive monoid
such that all its primes are free. Assume that $\mathbb P (M)$ has a
greatest element $p$ and that $[t,p]$ is a chain for every $t\in
\mathbb P (M)$. Then $Q_K(M)=Q_K(\mathbb P (M))$ is a von Neumann
regular ring and the natural monoid homomorphism
$$\psi \colon M\rightarrow \mon{Q_K(M)}$$ is an isomorphism.
\end{theor}

\begin{proof} Write $n_p:=|\rL(\mathbb P,p)|$ for $p\in \mathbb P:=\mathbb P(M)$.
We will use the notation of proposition \ref{posets}. For each $S\in
\mathcal S^0 (p)$ of the form $S=\{p_0,p_1,\dots ,p_s\}$, where
$p_0<p_1<\cdots < p_s=p$ is a maximal chain in $\mathbb P$, we
consider the $K$-algebra $Q(S)$ defined by $Q(S)=Q(M(S),
\sigma(k_1,\dots ,k_s))$, where $k_1=1$ and $k_2=n_{p_1}$, and in
general $k_{j}=k_{j-1}+n_{p_{j-1}}-1$, for $j=2,\dots ,s$. Indeed we
will find more useful to use simply the notation $Q(S)=Q(M(S),
(1,n_{p_1},\dots ,n_{p_{s-1}}))$ to denote this algebra. Recall that
in $Q(S)$ the $K$-endomorphisms $\sigma^p\colon L\to L$ used in
relation (\ref{eq:A.8}) are given by
$$\sigma^{p_1}(t_{\ell})=t_{\ell},\qquad
\sigma^{p_j}(t_{\ell})=t_{\ell+n_{j-1}-1} \quad (2\le j\le s).$$
We are going to give a corresponding definition of $Q(T)$ for
every $T\in \mathcal S ^i(p)$. Take any $T\in \mathcal S^i(p)$.
Then $T$ is a subset of $\mathbb P$ of the form
$$T=\{p_i,p_{i+1},\dots ,p_r\}\sqcup \{q\in \mathbb P:q<
p_{i}\},$$ where $p_i<p_{i+1}<\cdots <p_r=p$ is a chain that
cannot be refined. (This uses of course our hypothesis that
$[t,p]$ is a chain for every $t$ in $\mathbb P$.) Define
$$Q(T)=Q_K(M(T), (1,n_{p_{i+1}},\dots ,n_{p_{r-1}})),$$ where the
$K$-algebra $Q_K(M(T), (1,n_{p_{i+1}},\dots ,n_{p_{r-1}}))$ is the
one defined in \ref{defi: Q(M)} with the only difference that we use
for $\sigma^{p_{i+1}},\sigma^{p_{i+2}},\dots ,\sigma^{p_r}$ the
$K$-endomorphisms of $L$ defined by:
\begin{equation} \label{eq:B.12}
\sigma^{p_{i+1}}(t_{\ell})=t_{\ell},\qquad
\sigma^{p_j}(t_{\ell})=t_{\ell+n_{j-1}-1} \quad (i+2\le j\le r).
\end{equation}
All the $\sigma^q$ for $q\le p_i$ are the same as in the
definition of $Q_K(M(T))$, and agree with the ones used for these
primes in the definition of $Q_K(\mathbb P(M))$, i.e.
$\sigma^q(t_{\ell})=t_{\ell+n_q-1}$ when $q\le p_i$.

We are going to prove by induction on $i$ the following statement:

\medskip

\noindent (S$_i$) {\it For every $0\le i\le r$ and every $T\in
\mathcal S^i(p)$, the $K$-algebra $Q(T)$ is von Neumann regular
and $\mon{Q(T)}\cong M(T)$. In particular $Q(T)$ is a strongly
separative von Neumann regular ring.}

\medskip

For $i=0$, this is given by Proposition \ref{comparisonQMr}.
Assume that $r\ge i>0$ and that statement (S$_{i-1}$) holds. We
are going to show (S$_i$). Let $\{T_1,\dots ,T_{k}\}$ be a
non-trivial class in $\mathcal S^{i-1}(p)$, and let
$$q_j<p_i<p_{i+1}<\cdots <p_r=p$$
be the chain corresponding to $T_j$, for $j=1, \dots ,k$. In order
to simplify notation, we will denote $n_{\ell}=n_{p_{\ell}}$ for
$i\le \ell\le r$, and similarly $\alpha _{ij}:=\alpha _{p_i,q_j}$,
$\beta _{ij}:=\beta_{p_i,q_j}$, $\alpha _{\ell}:=\alpha
_{p_{\ell},p_{\ell -1}}$, $\beta _{\ell}:=\beta _{p_{\ell},p_{\ell
-1}}$ for $i+1\le \ell \le r$. Let $T\in \mathcal S^i(p)$ be given
by
$$T=\{p_i,p_{i+1},\dots ,p_r\}\sqcup \bigsqcup _{j=1}^{k}\{t\in T_j:t< p_{i}\}.$$
Observe that $\rL(T,p_i)=\rL(\mathbb P,p_i)=\{q_1,\dots ,q_k\}$, so
that $k=n_i$, and $\rL(T,p_u)=\{p_{u-1}\}$ for $r\ge u\ge i+1$. By
Proposition \ref{pullfgr}, we see that $M(T)$ is the pullback of the
family of natural maps $M(T_j)\to M(\Lambda)$, where $\Lambda $ is
the chain $\{p_i,\dots ,p_r\}$. (In categorical terms, the family of
maps $\{M(T)\to M(T_j): j=1,\dots ,k\}$, defined by sending all
$q\in T$, with $q\le q_{\ell}$ for $\ell\ne j$, to $0$, are the
limit of the family of maps $\{M(T_j)\to M(\Lambda): j=1,\dots ,k\}$
in the category of monoids).

Set $Q:=Q(\Lambda ,(n_{i}+1,n_{i+1},\dots ,n_{r-1}))$. Let $I_j$
be the ideal of $Q(T_j)$ generated by $e(q_j)$. Then there is a
natural isomorphism $\widetilde{\pi}_j\colon Q(T_j)/I_j\to Q$
which is essentially the identity on generators involving
$p_{i+1}, \dots ,p_r$ and which, in level $i$, acts as follows:
\begin{equation}
\label{eq:6.16} \widetilde{\pi}_j(e(p_i)t_{\ell})=
\begin{cases}e(p_i)t_{\ell}\quad & \text{if $\ell <j$,} \\
               e(p_i)t_{\ell+1} \quad   &\text{if $\ell\ge j$,}
      \end{cases}
\end{equation}
\begin{equation}
\label{eq:6.17} \widetilde{\pi}_j(\alpha _{ij})=e(p_i)t_j.
\end{equation}
Let $\pi _j\colon Q(T_j)\longrightarrow Q$ be the map defined by
the composition of the canonical projection $Q(T_j)\to Q(T_j)/I_j$
and the isomorphism $\widetilde{\pi}_j$. Then $\pi _j$ are
surjective maps with kernel $I_j$. Now, for $1\le j\le k$,
consider the $K$-algebra morphism $\rho _j\colon Q(T)\to Q(T_j)$
which annihilates all $e(q_{\ell})$ with $\ell\ne j$, is defined
as the identity in levels $i+1,\dots ,r$ and also below $q_j$, and
on level $i$ is defined as follows
\begin{equation}
\label{eq:6.18} \rho _j(\alpha_{i\ell})=
\begin{cases}e(p_i)t_{\sigma _j(\ell)}\quad & \text{if $\ell \ne j$,} \\
             \alpha_{ij}\quad & \text{if $\ell =j$} ,
      \end{cases}
\end{equation}
\begin{equation}
\label{eq:6.19} \rho _j(\beta _{i\ell})=\begin{cases} 0\quad & \text{if $ \ell\ne j$,} \\
             \beta_{ij}\quad & \text{if $\ell =j$,}
      \end{cases}
\end{equation}
\begin{equation}
\label{eq:6.20} \rho _j(e(p_i)t_u)=e(p_i)t_{u+k-1}.
\end{equation}
It is easily seen that these assignments give a well-defined
morphism $\rho _j$, that is, all relations between generators in
Definition \ref{defi: Q(M)} are preserved and the images by $\rho
_j$ of maps in $\Sigma\cup \Sigma _1$ are invertible in $Q(T_j)$.
The homomorphism $\rho _j$ is clearly surjective. We have
$\pi_j\circ \rho _j=\pi _{j'}\circ \rho _{j'}$ for all $j,j'$, and
so we get a canonical map $\rho \colon Q(T)\to P$, where $P$ denotes
the limit (pullback) of the maps $\pi_j\colon Q(T_j)\to Q$. We want
to see that $\rho$ is an isomorphism. Observe that $P$ has an ideal
of the form $I:=I_1\times \cdots \times I_k$ and that $P/I\cong Q$
canonically. The idea to show that $\rho$ is an isomorphism is to
prove, using the results in Section \ref{QQQ}, that the same
structure holds in $Q(T)$.

Let $G_j$ be the ideal of $Q(T)$ generated by $e(q_j)$. Let $A_j=
T\dnw q_j=\{x\in T:x\le q_j\}$, which is a lower subset of $T$. By
construction of $T$, we have that $A_j\cap A_{\ell}=\emptyset $
for $j\ne \ell$. It follows from (the proof) of Proposition
\ref{prop:ideallattice}(2) that the sum $\sum _{j=1}^k G_j$ is a
direct sum. Observe that the kernel of the map $\pi _j\circ \rho
_j$ is precisely $\bigoplus _{j=1}^k G_j$.
 On the other hand, the kernel
of the map $\rho _j\colon Q(T)\to Q(T_j)$ is precisely the ideal
$\bigoplus _{\ell \ne j}G_{\ell}$, and it follows that $\rho _j$
induces an isomorphism $G_j\cong I_j$ for $j=1,\dots ,k$.

Hence we have a commutative diagram with exact rows:

\begin{equation}
\label{eq:diagQ(T)}
\begin{CD}
0 @>>> \prod_{j=1}^k G_j @>>> Q(T) @>>> Q  @>>> 0  \\
&   & @V\cong VV  @VV\rho V  @V=VV  &                              \\
0 @>>> \prod_{j=1}^k I_j @>>> P  @>>> Q  @>>> 0
\end{CD}
\end{equation}

\noindent and thus $\rho\colon Q(T)\to P$ is an isomorphism. It
follows from Theorem \ref{pullbacktheorem} and $(\text{S}_{i-1})$
that $Q(T)$ is a strongly separative von Neumann regular ring.

By induction hypothesis, (S$_{i-1}$) holds, so that $Q(T_j)$ is a
von Neumann regular ring and the natural map $M(T_j)\to
\mon{Q(T_j)}$ is an isomorphism, for $1\le j\le k$. In particular
we see that $K_0(I_j)\cong \mathbb Z$, a generator being
$[e(p_i,q_j)]=[e(q_j)]$.

Now let us show that every idempotent $e$ in $Q(T)$ such that
$Q(T)eQ(T)\nsubseteq \prod_{j=1}^k G_j$ must satisfy
$Q(T)eQ(T)=Q(T)e(p_u)Q(T)=Q(T)(\sum _{q\le p_u}e(q))Q(T)$ for some
$u=i,\dots ,r$. Indeed we have $Q(T)eQ(T)+\prod_{j=1}^k
G_j=Q(T)e(p_u)Q(T)$, for some $u=i,\dots ,r$, and so
$$e(p_u)\oplus f_1\sim m\cdot e\oplus f_2,$$ where $f_2$ is a
finite direct sum of basic idempotents in $G_j$, $j=1, \dots ,k$.
Thus $e(p_u)\sim e(p_u)\oplus f_2$ and we get
$$e(p_u)\oplus f_1 \oplus f_2 \sim m\cdot e\oplus f_2.$$
Since $Q(T)$ is strongly separative, we get $e(p_u)\oplus f_1\sim
 m\cdot e$ and so $Q(T)e(p_u)Q(T)=Q(T)eQ(T)$.

Therefore, we only need to check the conditions
(\ref{eq:pullback}) of Theorem \ref{pullbacktheorem} for the
idempotents $\sum _{q\le p_u}e(q)$, $i\le u\le r$, and indeed,
without loss of generality, we can assume that
$e=1=(1_{Q(T_1)},\dots ,1_{Q(T_k)})$. The diagrams are of the
form:

\begin{equation}
\label{equ:6.xxxxxxxxxxxxxx}
\begin{CD}
K_1(Q(T_j)) @>(\pi _j)_*>> K_1(Q) @>\partial _j>> K_0(I_j)
 @>0>> K_0(Q(T_j)),
\end{CD}
\end{equation}
By (\ref{eq:6.17}) and the definition of the connecting map, we
get $\partial _j([e(p_i)t_j])=-[e(p_i,q_j)]$. By using
(\ref{eq:6.16}) we see that, for $\ell \ne j$,  we have $(\pi
_{\ell})_*([e(p_i)t_j])=[e(p_i)t_j]$ if $j<\ell$ and $(\pi
_{\ell})_*([e(p_i)t_{j-1}])=[e(p_i)t_j]$ if $j\ge \ell+1$. Thus
for $j=1,\dots ,k$ we have
$$K_1(Q)=(\pi _j)_*(K_1(Q(T_j)))+\big(\bigcap _{\ell \ne j}(\pi _{\ell})_*(K_1(Q(T_{\ell})))\big).$$
It follows from Theorem \ref{pullbacktheorem} that the natural map
$M(T)\to \mon{Q(T)}$ is an isomorphism.

This concludes the proof of Theorem \ref{firstcase}.
\end{proof}

We turn now to the general case. In this case, the strategy is to
analyze how $M(\mathbb P)$ is obtained from the different monoids
$M(\mathbb F(p))$, where $p$ ranges over the maximal elements of
$\mathbb P$.

For a maximal element $p$ of a poset $\mathbb P$, the poset
$\mathbb P\dnw p$ has $p$ as a greatest element, so we can use the
construction in Proposition \ref{posets} to obtain a poset
$\mathbb F (p)$ and a surjective order-preserving map $\Psi\colon
\mathbb F (p)\to \mathbb P\dnw p$. The map $\Psi $ preserves
chains, that is if $S$ is a chain in $\mathbb F(p)$ then $\Psi $
restricts to a bijection from $S$ to $\Psi (S)$. Moreover it is
easy to see that the map $S\to \Psi (S)$ is a bijection from the
set of maximal chains of $\mathbb F(p)$ onto the set $\mathcal
S^0(p)$ of maximal chains of $\mathbb P\dnw p$. We also recall two
fundamental properties of $\mathbb F(p)$:

(1) For every $t\in \mathbb F(p)$ the interval $[t,p]$ is a chain
(Proposition 5.1(i)).

(2) For $t_1,t_2\in \mathbb F(p)$, if $\Psi ([t_1,p])=\Psi
([t_2,p])$ then $t_1=t_2$. (This follows directly from the
construction of $\mathbb F(p)$.)

\begin{lem}\label{familylemma}
For every $q\in \mathbb F(p)$, the map $\Psi$ induces a bijection
from $\rL(\mathbb F(p), q)$ onto $\rL(\mathbb P\dnw p,\Psi (q))$.
\end{lem}

\begin{proof}
Write $\mathcal T:=\mathcal T (\mathbb P\dnw p)$, where $\mathcal T
(\mathbb P\dnw p)$ is the quiver associated to the poset $\mathbb
P\dnw p$, see Definition \ref{defi: Q(M)}. For $q'\in \mathbb P\dnw
p$ there is a bijection betweeen $\Psi ^{-1}(q')$ and the set of
paths in $\mathcal T$ from $p$ to $q'$. If $\Psi (q)=q'$ then an
element in $\rL(\mathbb F(p),q)$ corresponds to an enlargement of
the path in $\mathcal T$ from $p$ to $q'$ corresponding to $q$, by
an arrow from $q'$ to an element in $\rL(\mathbb P\dnw p,q)$. This
gives the result.
\end{proof}

Since $\Psi (q')<\Psi (q)$ whenever $q'<q$ in $\mathbb F(p)$, the
map $\Psi\colon \mathbb F(p)\to \mathbb P \dnw p$ can be extended to
a surjective monoid homomorphism, denoted in the same way,
$$\Psi
\colon M(\mathbb F(p))\longrightarrow M(\mathbb P\dnw p)=M(\mathbb
P)\mid p \, .
$$

The {\it depth} of an element $q$ in a poset $\mathbb P$ is the
maximum length of a chain of the form $q=q_0<q_1<\cdots <q_t$. Let
us denote the poset of elements of $\mathbb P$ of depth $\le s$ by
$\mathbb P_s$.

\begin{prop}\label{primitivepushout}
Let $M$ be a finitely generated primitive monoid such that all
primes of $M$ are free. For a maximal element $p$ of $\mathbb P
(M)$, consider the monoid homomorphism $\Psi \colon M(\mathbb
F(p))\to M\mid p$ defined above. Assume that $p$ has height $r$. For
$0\le i\le r$, there are primitive monoids $M^i(p)$ such that, for
$0\le i< j\le r$, there are surjective monoid homomorphisms $\Psi
_{ji}\colon M^i(p)\to M^j(p)$ with $\Psi _{ji}=\Psi _{jk}\Psi _{ki}$
when $i< k< j$, such that the following properties are satisfied:
\begin{enumerate}[{\rm (i)}]
\item $M^0(p)=M(\mathbb F (p))$, $M^r(p)=M\mid p\, $ and $\, \Psi
_{r0}=\Psi$.
\item $\Psi _{i0}$ induces an isomorphism of posets from $(\mathbb P(M^0(p)))_{r-i}$ onto
$(\mathbb P (M^i(p)))_{r-i}$.
\item $\Psi _{ri}$ induces an isomorphism between $M^i(p)\mid q$ and
$M\mid \Psi _{ri}(q)$ for $q\in \mathbb P(M^i(p))$ of depth $\ge
r-i-1$.
\item $(M^i(p)\mid q_1)\cap (M^i(p)\mid q_2)=0$ for
two incomparable elements $q_1 $ and $q_2$ in $ \mathbb P(M^i(p))$
of depth $\le r-i-1$.
\end{enumerate}

\smallskip

\noindent The monoid $M^{i+1}(p)$ (respectively, $M$) is obtained
from the monoid $M^{i}(p)$ (respectively, {\rm $\prod _{p\in
{\text{Max}(\mathbb P)}}M\mid p$)} by a finite sequence of crowned
pushout diagrams of the form considered in Lemma
\ref{structVcrownpushout}.
\end{prop}

\begin{proof} We proceed by induction on $i$.

Observe that order-ideals in $M^0(p):=M(\mathbb F (p))$ are in
bijection with lower subsets of $\mathbb F (p)$, see Proposition
\ref{wellknownlattice}. If $q_1$ and $q_2$ are incomparable elements
of $\mathbb F (p)$ then $(\mathbb F(p)\dnw q_1)\cap (\mathbb
F(p)\dnw q_2)=\emptyset$, because $[t,p]$ is a chain for every $t\in
\mathbb F (p)$. Hence property (iv) holds for $M^0(p)$. Property
(ii) holds vacuously, and property (iii) follows from Lemma
\ref{familylemma}.

Assume that $M^i(p)$ and $\Psi _{i0}$, $\Psi _{ri}$ have been
constructed, where $0\le i<r-1$, so that properties (i)--(iv) hold.
We need to build $M^{i+1}(p)$ so that the map $\Psi _{ri}\colon
M^i(p)\to M\mid p$ factors through $M^{i+1}(p)$. Write $\mathbb
P^i:=\mathbb P (M^i(p))$. Let $q$ be an element in $\mathbb P ^i$ of
depth $r-i-2$. Consider the set $\rL(\mathbb P ^i,q)=\{q_1,\dots
,q_k\}$ of lower covers of $q$ in $\mathbb P^i$. Observe that the
elements $q_u$ have depth exactly $r-i-1$. Indeed, $q_u$ has depth
$\le r-i-1$, and thus by (ii) $[q_u,p]\cong [\Psi
_{i0}^{-1}(q_u),p]$, so that $[q_u,p]$ is a chain. It follows that
the depth of $q_u$ is exactly $r-i-1$.

Since $q_1, \dots ,q_k$ are mutually incomparable, it follows from
(iv) that $(M^i(p)\mid q_u)\cap (M^i(p)\mid q_v)=0$ for $u\ne v$.
Since the depth of $q_j$ is $r-i-1$, it follows from (iii) that
$${\Psi _{ri}}_{|(M^i(p)\mid q_j)}\colon
M^i(p)\mid q_j \longrightarrow M\mid \Psi _{ri}(q_j)$$ is an
isomorphism, for $j=1,\dots ,k$.

Consider the order-ideals of $M^i(p)$ given by:
$$Z_u:=((\Psi _{ri})_{|(M^i(p)\mid q_u)})^{-1}\big(\Psi _{ri}(M^i(p)\mid q_1)\bigcap
\Psi_{ri}(M^i(p)\mid q_2)\big)$$ for $u=1,2$. The map
$$\varphi=\big( ((\Psi _{ri})_{|(M^i(p)\mid q_2)})^{-1}\circ
(\Psi _{ri})_{|(M^i(p)\mid q_1)}\big)_{|Z_1}\colon
Z_1\longrightarrow Z_2$$ is a monoid isomorphism, and $Z_u\subseteq
(M^i(p)\mid q_u)$ for $u=1,2$. Observe that, by Lemma
\ref{familylemma},  condition (ii), and the fact that $\Psi =\Psi
_{ri}\Psi_{i0}$, it follows that $\Psi _{ri}$ induces a bijection
from $\rL(\mathbb P^i,q)$ onto $\rL(\mathbb P\dnw p,\Psi _{ri}(q))$.
This implies that $Z_u$ is strictly contained in $M^i(p)\mid q_u$
for $u=1,2$.

By (iv) we have $Z_1\cap Z_2=0$. So we can consider the crowned
pushout $M'$ of $(M^i(p),Z_1,Z_2,\varphi)$, see Section
\ref{sect:pushout}. By Lemma \ref{structVcrownpushout}, $M'$ is a
primitive monoid with all primes free, and $\mathbb P (M')=\mathbb
P^i\setminus \mathbb P (Z_2)$ with the order structure given by the
restriction of the order in $\mathbb P ^i$ and the additional
relations, for $p\in \mathbb P (Z_1)$ and $q\in \mathbb P
^i\setminus (\mathbb P(Z_1)\cup \mathbb P (Z_2))$, given by  $p<q$
if $\varphi (p)<q$ in $\mathbb P ^i$.  Moreover, the map $\Psi
_{ri}\colon M^i(p)\to M$ factors as
\begin{equation}
\label{equ:xxx}
\begin{CD}
M^i(p) @>{\mu '}>>  M' @>{\lambda '}>> M ,
\end{CD}
\end{equation} where the map $\mu '\colon M^i(p)\to M'$
can be identified by Lemma \ref{structVcrownpushout} with the
natural identification map sending $p$ and $\varphi (p)$ to $p$ for
$p\in \mathbb P (Z_1)$, and sending $p$ to $p$ for $p\in \mathbb P
^i\setminus (\mathbb P (Z_1)\cup \mathbb P (Z_2))$. Moreover, the
map $\lambda '$ induces an isomorphism from $M'\mid \{q_1,q_2\}$
onto $M\mid \{\Psi _{ri}(q_1),\Psi _{ri}(q_2)\}$. Proceeding in this
way, we obtain a new primitive monoid $M''$, such that all primes of
$M''$ are free, with a factorization of the map $\Psi _{ri}\colon
M^i(p)\to M$  of the form
\begin{equation}
\label{equ:xxy}
\begin{CD}
M^i(p) @>{\mu ''}>>  M'' @>{\lambda ''}>> M ,
\end{CD}
\end{equation}
such that $\lambda'' $ induces an isomorphism from $M''\mid
\{q_1,\dots ,q_k\}$ onto $M\mid \{\Psi _{ri}(q_1),\dots
,\Psi_{ri}(q_{k})\}$, and such that $\mu ''$ is the identity on all
the primes of depth $\le r-i-1$. Clearly the map $\lambda ''$
induces an order-isomorphism from $\mathbb P (M'')\dnw  q$ onto
$\mathbb P\dnw \Psi _{ri}(q)$, and thus a monoid isomorphism from
$M''(p)\mid q$ onto $M\mid \Psi _{ri}(q)$.

Proceeding in this way with all the primes $q$ in $M^i(p)$ of
depth $r-i-2$, we get a factorization of the map $\Psi _{ri}$ as
\begin{equation}
\label{equ:xyx}
\begin{CD}
M^i(p) @>{\Psi_{i+1,i}}>>  M^{i+1}(p) @>{\Psi_{r,i+1}}>> M\mid p ,
\end{CD}
\end{equation}
with the desired properties. Note that for $i=r-1$, we get an
isomorphism \linebreak $\Psi _{r,r-1}\colon M^{r-1}(p)\to M\mid
p$.

Finally observe that a similar process can be used to get $M$ from
$\prod _{p\in {\text{Max}(\mathbb P)}}M\mid p$ by a finite
sequence of crowned pushouts.
\end{proof}

Let $\mathbb P$ be a poset, and let $A$ be a lower subset of
$\mathbb P$. Put
$$\partial A=\{p\in\mathbb P : \rL(\mathbb P,p)\cap A\ne \emptyset \} \cup A.$$

Observe that, as explained at the beginning of this section (Step
2), the next theorem, together with Theorem \ref{firstcase} and
Proposition \ref{primitivepushout}, completes the proof of Theorem
\ref{main}, because the basic steps used to build $M(\mathbb P)$
from the monoids $M(\mathbb F (p))$, for $p\in \text{Max}(\mathbb
P)$, are of the form described in the theorem (and moreover Theorem
\ref{main} holds for the monoids $M(\mathbb F (p))$ by Theorem
\ref{firstcase}).

\begin{theor}
\label{regpushout} Let $M$ be a finitely generated primitive monoid
with all primes free. Assume that $\mathbb P:=\mathbb P (M)$
contains two lower subsets $A,A'$ such that $\partial A \cap
\partial A'=\emptyset $, and assume moreover that there is an
order-isomorphism $\varphi \colon A\to A'$. Let $\varphi \colon
M(A)\to M(A')$ denote the induced monoid isomorphism, and let $M'$
be the crowned pushout of $(M,M(A),M(A'),\varphi)$.

Suppose that $Q_K(M)$ is von Neumann regular and that the canonical
map $\psi _M\colon M\to \mon{Q_K(M)}$ is an isomorphism. Then
$Q_K(M')$ is also von Neumann regular and the map $\psi _{M'}\colon
M'\to \mon{Q_K(M')}$ is an isomorphism.
\end{theor}

\begin{proof}
Obviously we can assume that the order-isomorphism $\varphi \colon
A\to A'$ respects the labelling of the arrows of the quivers
$\mathcal T (A)$ and $\mathcal T (A')$ associated with $A$ and $A'$
respectively, see Definition \ref{defi: Q(M)}. Consequently we get a
$K$-algebra isomorphism,  denoted in the same way, $\varphi\colon
Q_K(A)\to Q_K(A')$.

We have a natural isomorphism, still denoted by $\varphi$, from
$e(A)Q_K(M)e(A)$ onto $e(A')Q_K(M)e(A')$ given by the composition
of the isomorphisms
$$e(A)Q_K(M)e(A)\longrightarrow Q_K(A)\longrightarrow
Q_K(A')\longrightarrow e(A')Q_K(M)e(A'),$$ where the first and third
maps are isomorphisms by Theorem \ref{theor:repres}. It follows that
$I(A)$ and $I(A')$ are Morita-equivalent (where $I(A)$ (resp.
$I(A')$) denotes the ideal of $Q_K(M)$ generated by $e(A)$ (resp.
$e(A')$)). Write $Q:=Q_K(M)$. Take $Q_1=Q/I(A')$ and $Q_2=Q/I(A)$,
and consider the $Q_1$-$Q_2$-bimodule $N=Qe(A)\otimes_{e(A)Qe(A)}
e(A')Q$ and the $Q_2$-$Q_1$-bimodule $M=Qe(A')\otimes
_{e(A')Qe(A')}e(A)Q$, where $e(A)Qe(A)$ acts on $e(A')Q$ by $x\cdot
y=\varphi (x)y$ for $x\in e(A)Qe(A)$ and $y\in e(A')Q$, and
similarly for the action of $e(A')Qe(A')$ on $e(A)Q$. Let $U$ be the
ring described in the proof of Proposition \ref{crownpushoutprop},
so that $U$ is the ring of all matrices $X=\begin{pmatrix}q+I(A') &
n
\\ m & q+I(A)
\end{pmatrix}$ in $\begin{pmatrix}Q_1 & N \\ M & Q_2\end{pmatrix}$.
By Proposition \ref{crownpushoutprop} there is an injective unital
$K$-algebra morphism $\omega \colon Q\to U$ given by $\omega
(q)=\text{diag}(q+I(A'),q+I(A))$. Moreover $U$ is von Neumann
regular, and $\mon{U}\cong M'$ in a natural way, and we have an
isomorphism $Q/(I(A)+I(A'))\cong U/J$, where $$J=\begin{pmatrix}
I(A) & N \\ M & I(A')
\end{pmatrix}.$$

Set $e=1_U-\text{diag}(0,e(A'))\in U$, and observe that $e$ is a
full idempotent in $U$, that is, $UeU=U$. It follows that
$\mon{eUe}\cong M'$.
 We are going to show that
there is an isomorphism $\delta \colon Q_K(M')\to eUe$ such that
the induced map $\mon{\delta ^{-1}}\colon M'\cong \mon{eUe}\to
\mon{Q_K(M')}$ is the canonical map $M'\to \mon{Q_K(M')}$.

We define the map $\delta$ on the canonical generators of
$Q_K(M')$, given in Definition \ref{defi: Q(M)}. The reader can
easily show that the defining relations are satisfied in $eUe$ and
that the map so defined  $\mathcal A \to eUe$ is $(\Sigma\cup
\Sigma _1)$-inverting.

By Lemma \ref{structVcrownpushout},  $\mathbb P ':=\mathbb P (M')=
\mathbb P \setminus A'$, with the order relation induced by the
order relation in $\mathbb P$ and the additional relations $q<p$
whenever $q\in A$, $p\in \mathbb P'\setminus A$, and $\varphi (q)<p$
in $\mathbb P$. Define $\delta (e(p))=\omega (e(p))\in U$ for $p\in
\mathbb P'$. Observe that $\{\omega (e(p)): p\in \mathbb P'\}$ is an
orthogonal family of idempotents in $eUe$ with sum $e$. Now for
$q,p\in \mathbb P '$ such that $q\in \rL(\mathbb P,p)$, define
$\delta (e(p,q))=\omega (e(p,q))$ and $\delta (\alpha_{p,q})=\omega
(\alpha _{p,q})$ and $\delta (\beta_{p,q})=\omega (\beta _{p,q})$.
If $q\in A$, $p\in \mathbb P'\setminus A$ and $\varphi (q)\in
\rL(\mathbb P,p)$, then define
$\delta (e(p,q))=\omega (e(p,\varphi (q)))=\begin{pmatrix} 0 & 0 \\
0 & e(p,\varphi (q))\end{pmatrix}$, and $\delta (\alpha
_{p,q})=\omega (\alpha _{p,\varphi (q)})$, and
$$\delta (\beta_{p,q})=\begin{pmatrix}0 & 0 \\ \beta _{p,\varphi (q)}\otimes e(q) &
0\end{pmatrix}.$$
 Note that this is
well-defined by our hypothesis that $\partial A\cap
\partial A'=\emptyset$.

Now it is straightforward to show that $\delta $ is an
isomorphism. Indeed the induced map $ Q_K(M')/I_{Q_K(M')}(A)\to
eUe/eJe=Q/(I(A)+I(A'))$ is an isomorphism by Remark
\ref{remark:quotients}, using our assumption that $\partial A\cap
\partial A'=\emptyset$. The restriction map $\ol{\delta}\colon I_{Q_K(M')}(A)\to eJe$ is also an
isomorphism. For, note that elements in $Q_K(M')_{(\gamma _1,\gamma
_2)}$ with $r(\gamma _1)=r(\gamma _2)\in A$ can be classified in
four classes, depending on whether each $\gamma _i$ contains   an
arrow $(p,q)$ with $q\in A$, $p\in \mathbb P '\setminus A$ and
$\varphi (q)\in \rL(\mathbb P,p)$, or does not contain it. Each of
the four classes
corresponds to a corner in $$eJe=\begin{pmatrix} I(A) & Qe(A)\otimes _{e(A)Qe(A)}e(A')Q(1-e(A')) \\
(1-e(A'))Qe(A')\otimes _{e(A')Qe(A')}e(A)Q &
(1-e(A'))I(A')(1-e(A'))
\end{pmatrix}.$$
Thus, one obtains from Proposition \ref{prop:ideallattice} that
$\ol{\delta}$ is surjective. In order to show injectivity of
$\ol{\delta}$, take a nonzero element $x$ in $I_{Q_K(M')}(A)$. By
Lemma \ref{lem:injectivity},  there exist $p\in A$ and $z_1,z_2\in
Q_K(M')$ such that $z_1xz_2$ has the trivial pair of paths $(p,p)$
in the support, and all the other elements in the support of
$z_1xz_2$ are of the form $(\gamma _1,\gamma _2)$, where $\gamma
_1$ and $\gamma _2$ are paths in $\mathcal T(A)$ starting in $p$
and ending in a common vertex. Then we have that $\delta
(z_1xz_2)\ne 0$, because the component in $Q_K(M')_{(p,p)}$ of
$z_1xz_2$ is sent by $\delta$ to a nonzero diagonal element in
$eJe$.

Therefore the map $\ol{\delta}$ is an isomorphism and consequently
so is the map $\delta$, as desired.
\end{proof}

\section{Acknowledgements}
I am grateful to Ken Goodearl for his useful comments on a
preliminary version of this paper, and to the anonymous referee for
his or her very careful reading of the paper and stimulating
suggestions.

\end{document}